\RequirePackage{fix-cm}
\documentclass[smallextended]{svjour3}       
\smartqed  
\usepackage{graphicx}
\usepackage{amsmath}
\usepackage{amsfonts}
\usepackage{amsmath}
\usepackage{amssymb}
\usepackage{graphics}
\numberwithin{equation}{section}
\allowdisplaybreaks

\newcommand{\al}{\alpha}
\newcommand{\be}{\beta}
\newcommand{\ga}{\gamma}

\newcommand{\de}{\delta}
\newcommand{\De}{\Delta}
\newcommand{\om}{\omega}

\renewcommand{\th}{\theta}

\newcommand{\R}{\ensuremath{{\mathbb R}}}

\newcommand{\brmk}{\begin{rmk}}
\newcommand{\ermk}{\end{rmk}}
\newcommand{\partref}[1]{\hbox{(\csname @roman\endcsname{\ref{#1}})}}

\newcommand{\beq}{\begin{equation}}
\newcommand{\eeq}{\end{equation}}
\newcommand{\beqs}{\begin{equation*}}
\newcommand{\eeqs}{\end{equation*}}
\newcommand{\beqa}{\begin{equation}\begin{aligned}}
\newcommand{\eeqa}{\end{aligned}\end{equation}}
\newcommand{\beqas}{\begin{equation*}\begin{aligned}}
\newcommand{\eeqas}{\end{aligned}\end{equation*}}

\newcommand{\half}{\frac{1}{2}}

\newcommand{\eps}{\varepsilon}

\newcommand{\supp}{\text{supp}}

\numberwithin{equation}{section}
\numberwithin{theorem}{section}
\numberwithin{lemma}{section}
\numberwithin{proposition}{section}
\numberwithin{corollary}{section}
\numberwithin{definition}{section}
\numberwithin{remark}{section}
\numberwithin{example}{section}

\begin{document}

\title{Vanishing Viscosity Approach to the Compressible Euler Equations
for Transonic Nozzle and Spherically Symmetric Flows}


\titlerunning{Vanishing Viscosity Approach to the Compressible Euler Equations}        

\author{Gui-Qiang G. Chen \and \\
Matthew R. I. Schrecker
}


\institute{Gui-Qiang G. Chen  \at
Mathematical Institute, University of Oxford,
Oxford, OX2 6GG, UK \\
              \email{chengq@maths.ox.ac.uk}           
           \and
           Matthew R. I. Schrecker \at
              Mathematical Institute, University of Oxford,
Oxford, OX2 6GG, UK\\
\email{matthew.schrecker@gmail.com}
}

\date{Received: November 28, 2017 / Accepted: February 16, 2018}

\maketitle

\begin{abstract}
We are concerned with globally defined entropy solutions to the Euler equations
for compressible fluid flows in transonic nozzles with general cross-sectional areas.
Such nozzles include the de Laval nozzles and other more general nozzles
whose cross-sectional area functions are allowed at the nozzle ends to
be either zero (closed ends) or infinity (unbounded ends).
To achieve this, in this paper, we develop a vanishing viscosity method to construct globally defined
approximate solutions and then establish essential uniform estimates in weighted $L^p$ norms
for the whole range of physical adiabatic exponents $\gamma\in (1, \infty)$, so that the
viscosity approximate solutions satisfy the general $L^p$ compensated compactness
framework.
The viscosity method is designed to
incorporate artificial viscosity terms with the natural Dirichlet boundary conditions
to ensure the uniform estimates.
Then such estimates lead to both the convergence of the approximate solutions and
the existence theory of globally defined finite-energy entropy solutions
to the Euler equations for transonic flows that may have different end-states
in the class of nozzles with general cross-sectional areas for all $\gamma\in (1, \infty)$.
The approach and techniques developed here apply to other problems with similar difficulties.
In particular,
we successfully apply them to construct globally defined spherically symmetric entropy solutions
to the Euler equations for all $\gamma\in (1, \infty)$.
\end{abstract}



\section{\, Introduction}
\label{intro}
We are concerned with globally defined entropy solutions to the Euler equations for
transonic nozzle flows with general cross-sectional areas and related
compressible flows of geometric structure including spherically symmetric flows.
Then the Euler equations
may be reduced to the following one-dimensional Euler system with geometric
terms:
\begin{equation}\label{transonic}
\begin{cases}
\rho_t+ m_x+\frac{A'(x)}{A(x)}m=0, \\[1mm]
m_t+\big(\frac{m^2}{\rho} +p(\rho)\big)_x +\frac{A'(x)}{A(x)}\frac{m^2}{\rho}=0
\end{cases}
\end{equation}
for $(t,x)\in\mathbb{R}_+\times\mathbb{R}$ with $\R_+=[0, \infty)$,
where  $\rho:\mathbb{R}\rightarrow\mathbb{R}_+$ is the density of the fluid,
$u=\frac{m}{\rho}:\mathbb{R}\rightarrow\mathbb{R}$ is the velocity,
and $A=A(x)$ is a given $C^2$--function $A:\mathbb{R}\rightarrow\mathbb{R}_+$.
The equation of state is assumed to be that of a polytropic gas:
\begin{equation*}
 p(\rho)=\kappa\rho^\gamma,  \qquad \gamma>1,
\end{equation*}
where
$\kappa:=\frac{(\gamma-1)^2}{4\gamma}$ by scaling (without loss of generality).

\medskip
We first focus on
the Cauchy problem for system \eqref{transonic}:
\begin{equation}\label{prob:Cauchy}
(\rho, m)|_{t=0}=(\rho_0, m_0),
\end{equation}
where the initial data $(\rho_0, m_0)(x)$ have end-point states $(\rho_\pm,m_\pm)$ at $x=\pm\infty$, for constants
$\rho_\pm\geq0$ and $m_\pm\in\mathbb{R}$.
The end-states are allowed to be different, {\it i.e.},  $(\rho_+,m_+)\ne (\rho_-,m_-)$ in general.

The Cauchy problem \eqref{transonic}--\eqref{prob:Cauchy}
models transonic nozzle flow through a variable-area
duct, where $A(x)$ describes the cross-sectional area of the nozzle at point $x\in \R$
({\it cf}. \cite{CourantFriedrichs,EM,GL,GM,L1,L2,L3,Whitham}).
The existence of global solutions of this problem was first obtained in \cite{L1}
by incorporating
the steady-state building
blocks with the Glimm scheme \cite{Gl}, provided that the initial data
have small total variation and are bounded away from both sonic and vacuum
states.
Some numerical methods were introduced to compute transient
gas flows in a de Laval nozzle
in \cite{GL,GM}.
A mathematical analysis of the qualitative
behavior of nonlinear waves for nozzle flow  was given in \cite{L3}.
On the other hand,
the general transonic nozzle problem, so that the initial data are allowed to
be arbitrarily large with different end-states and only relative finite-energy
and that the cross-sectional areas are allowed
to tend to either zero (closed ends) or infinity (unbounded ends) as $x\to \pm \infty$,
has not been understood yet, including the existence theory for all $\gamma\in (1,\infty)$,
since the solution may blow up at the closed ends.
For the general case,
the corresponding cross-sectional area function $A(x)>0$ is
a $C^2$--function that requires only either
\begin{equation}\label{1.3a}
\big\|\frac{A'}{A}\big\|_{L^\infty(\R)}+\|A'\|_{L^1(-\infty,0)}<\infty
\end{equation}
or
\begin{equation}\label{1.3b}
\big\|\frac{A'}{A}\big\|_{L^\infty(\R)}+\|A'\|_{L^1(0, \infty)}<\infty.
\end{equation}
In particular, this general class of nozzles includes the de Laval nozzles with closed ends.
In this paper, we first establish the existence of globally defined entropy solutions to the
Euler equations \eqref{transonic}
for compressible fluid flows in transonic nozzles with general cross-sectional areas
satisfying \eqref{1.3a} or \eqref{1.3b}.
To achieve this,
we first develop a viscosity method to construct globally defined
approximate solutions
and then establish several essential uniform estimates for the whole interval
of adiabatic exponents $\gamma\in (1, \infty)$, so that the
viscosity approximate solutions satisfy the general $L^p$ compensated compactness
framework established in Chen-Perepelitsa \cite{ChenPerep1} (also see \cite{Chen2} and the references cited therein).
Our viscosity method is carefully designed to
incorporate artificial viscosity terms with the Dirichlet boundary conditions
naturally to ensure the uniform estimates for all $\gamma\in (1, \infty)$.
Then these uniform estimates lead to both the convergence of the approximate solutions and
the existence theory of globally defined entropy solutions of problem
\eqref{transonic}--\eqref{prob:Cauchy}
for initial data of relative finite-energy for all $\gamma\in (1,\infty)$,
provided that
the cross-sectional area function $A(x)$ satisfies \eqref{1.3a} or \eqref{1.3b}.

The case that $A:\mathbb{R}_+\rightarrow\mathbb{R}_+$ with $A(x)=\omega_n x^{n-1}, \omega_n>0$, for $x\in \R_+$
corresponds to the equations for the isentropic
Euler equations with spherical symmetry, which does not satisfy condition \eqref{1.3a} or \eqref{1.3b}.
The study of spherically symmetric solutions dates back to the 1950s, and is motivated
by many important physical problems such as flow in a jet engine inlet manifold
and stellar dynamics including gaseous stars and supernovae
formation
({\it cf}. \cite{CourantFriedrichs,Guderley,LiWang,Rosseland,Slemrod,Whitham}).
The central feature is the strengthening of waves as they move radially inward,
especially near the origin.
Various
evidence indicates that spherically symmetric solutions of the compressible Euler
equations may blow up near the origin at a certain time in some situations. A
longstanding open, fundamental problem is whether a concentration could be formed at
the origin, that is, the density becomes a delta measure at the origin, especially when a
focusing spherical shock is moving inward towards the origin ({\it cf.} \cite{CourantFriedrichs,Rosseland,Whitham}).
The first existence result for global entropy
solutions including the origin was established in Chen \cite{Chen} for a class of $L^\infty$
Cauchy data of arbitrarily large amplitude, which models outgoing blast waves and large-time
asymptotic solutions; some further extensions of such Cauchy initial data for global entropy
solutions in $L^\infty$
can be found in Huang-Li-Yuan \cite{HuangLiYuan}
and the references cited therein.
Also see Slemrod \cite{Slemrod} for the resolution of the spherical piston
problem for isentropic gas dynamics via a self-similar viscous limit,
and LeFloch-Westdickenberg \cite{LeFlochWest} for a compactness framework to ensure the strong compactness of
spherically symmetric approximate solutions with uniform finite-energy norms for the
case $1<\gamma\le \frac{5}{3}$.
In Chen-Perepelitsa \cite{ChenPerep2}, under the assumption that $\gamma\in (1, 3]$,
the existence of globally defined spherically symmetric entropy solutions
was established via the vanishing viscosity approximation.
As a direct application of the approach and techniques developed in this paper
for the general transonic nozzle problem,
we successfully remove the restriction that $\gamma\in (1, 3]$ and establish
the existence of globally defined finite-energy entropy
solutions for the general case of large initial data for the full range $\gamma>1$,
including the unsolved case $\gamma>3$.
This implies that there exist globally defined spherically symmetric solutions
with large initial data of finite-energy for the whole range $\gamma\in (1,\infty)$,
which do not form concentrations at the origin.

The structure of the paper is as follows:
In \S 2, we discuss the properties of weak entropy functions of system \eqref{transonic}, introduce the definition of entropy solutions,
and present the main theorems, Theorems  \ref{thm:main}--\ref{thm:sph-symm}.

In \S \ref{sec:artificial}--\S 5,
we develop the viscosity method to construct the afore-mentioned approximate solutions and
demonstrate their convergence
to entropy solutions of the Cauchy problem \eqref{transonic}--\eqref{prob:Cauchy},
first under the following assumptions on the cross-sectional area function $A$:
\begin{align}
& 0< A_0\leq A(x)\leq A_1<\infty,\label{1.4}\\
& \|A\|_{C^2(\mathbb{R})}+\|A'\|_{L^1(\mathbb{R})}\leq A_2<\infty. \label{1.5}
\end{align}
The construction of the approximate solutions in \S \ref{sec:artificial} employs the standard theory of quasilinear parabolic systems
as detailed in \cite{LadySoloUral}.
In \S 4, we derive  \it uniform \rm estimates, independent of the viscosity coefficient $\varepsilon$,
on the approximate solutions, so that they satisfy the requirements of
the general $L^p$ compensated compactness framework in \cite{ChenPerep1} to pass to the limit
as $\varepsilon\rightarrow 0$, and hence deduce, in \S 5, the existence
of globally defined entropy solutions of the Cauchy problem \eqref{transonic}--\eqref{prob:Cauchy}
with general large initial data of different end-states and relative finite-energy in the sense of Theorem \ref{thm:main}.

In \S \ref{sec:general}, we describe how the additional assumptions on $A(x)$ in \eqref{1.4}--\eqref{1.5}
can be removed to obtain
the expected results in Theorem \ref{thm:main} for the general case \eqref{1.3a} or \eqref{1.3b}.

Finally, in \S 7, we show how the approach and techniques developed in \S 3--\S 6
can be extended to the general setting of spherically
symmetric solutions to the Euler equations, leading to our second main theorem, Theorem \ref{thm:sph-symm},
for the whole range of adiabatic exponents $\gamma\in (1, \infty)$, especially including the unsolved case: $\gamma>3$.

\section{\, Entropy Solutions and Main Theorems}\label{sec:2}

In this section, we discuss the properties of weak entropy functions of system \eqref{transonic},
introduce the definition of entropy solutions,
and present the main theorems, Theorems  \ref{thm:main}--\ref{thm:sph-symm}.

\subsection{\, Entropy}
An entropy-entropy flux pair (or entropy pair, for simplicity) is a pair
of functions $(\eta,q):\mathbb{R}_+\times\mathbb{R}\rightarrow\mathbb{R}^2$ such that
\begin{equation}\label{2.1}
 \nabla q(\rho,m)=\nabla\eta(\rho,m)\nabla\begin{pmatrix}
                                 m\\
                                 \frac{m^2}{\rho}+p(\rho)
                                \end{pmatrix},
\end{equation}
where $\nabla$ is the gradient with respect to the conservative variables $(\rho, m)$.
For example, the mechanical energy and its flux form an entropy pair:
\begin{eqnarray}
&&\eta^*(\rho,m)=\frac{1}{2}\frac{m^2}{\rho}+\frac{\kappa}{\gamma-1}\rho^\gamma,\label{2.2a}\\
&& q^*(\rho,m)=\frac{1}{2}\frac{m^3}{\rho^2}+\frac{\kappa\gamma}{\gamma-1}m\rho^{\gamma-1}. \label{2.2b}
\end{eqnarray}

We also use the well-known fact ({\it cf.} \cite{DiPerna,LPT})
that any weak entropy pair (that is, the corresponding entropy
vanishes at $\rho=0$) for system \eqref{transonic} can be expressed as
\begin{align}
\eta^\psi(\rho,m)=&\,\rho\int_{-\infty}^\infty\psi(\frac{m}{\rho}+\rho^\theta s)[1-s^2]_+^{\frac{3-\gamma}{2(\gamma-1)}}\, ds,\label{2.3a} \\
q^\psi(\rho,m)
=&\,\rho\int_{-\infty}^{\infty}(\frac{m}{\rho}+\theta\rho^\theta s)\psi(\frac{m}{\rho}+\rho^\theta s)[1-s^2]_+^{\frac{3-\gamma}{2(\gamma-1)}} \,ds, \label{2.3b}
\end{align}
where $\theta=\frac{\gamma-1}{2}$,
and $\psi:\mathbb{R}\rightarrow\mathbb{R}$ is said to
be a generator of the entropy pair $(\eta^\psi, q^\psi)$.

\subsection{\, Entropy Solutions}
Now we introduce the definition of entropy solutions of the Cauchy problem \eqref{transonic}--\eqref{prob:Cauchy}.

\begin{definition}\label{def:entropysol}
\, An {\it entropy solution} \rm of the Cauchy problem \eqref{transonic}--\eqref{prob:Cauchy} is a pair
  $(\rho,m):\mathbb{R}^2_+\rightarrow\mathbb{R}^2_+$ such that
\begin{enumerate}
\item[(i)] For any $\phi\in C_c^\infty(\mathbb{R}^2_+)$,
\begin{eqnarray*}
&&\int_{\mathbb{R}^2_+}(\rho\phi_t+m\phi_x)A(x)\,dx\,dt+\int_\mathbb{R}\rho_0(x)\phi(0,x)A(x)\,dx=0,
\label{2.4a}\\
&& \int_{\mathbb{R}^2_+}\big(m\phi_t+\frac{m^2}{\rho}\phi_x+p(\rho)\frac{(A\phi)_x}{A}\big)A\,dx\,dt+\int_\mathbb{R}m_0(x)\phi(0,x)A(x)\,dx=0;
\label{2.4b}
\end{eqnarray*}
\item[(ii)] For any convex $\psi(s)$, with sub-quadratic growth at infinity, generating the entropy pair $(\eta^\psi,q^\psi)$,
 \begin{equation}\label{entropyinequality}
   \big(\eta^\psi A(x)\big)_t+\big(q^\psi A(x)\big)_x+A'(x)\big(m\eta^\psi_\rho+\frac{m^2}{\rho}\eta^\psi_m-q^\psi\big)\leq 0
 \end{equation}
  in the sense of distributions.
 \end{enumerate}
\end{definition}

In order to define the energy with respect to the end-states $(\rho_\pm, m_\pm)=(\rho_\pm, \rho_\pm u_\pm)$,
where $u_\pm=\frac{m_\pm}{\rho_\pm}$ if $\rho_\pm\neq 0$, or $u_\pm=0$ otherwise,
we follow
\cite{ChenPerep1} in defining smooth, monotone functions $(\bar{\rho}(x),\bar{u}(x))$
such that, for some $L_0>1$,
\begin{equation}\label{2.5}
(\bar{\rho}(x),\bar{u}(x))=
\begin{cases}
(\rho_+,u_+), \quad & x\geq L_0,\\[1mm]
(\rho_-,u_-), \quad & x\leq -L_0.
 \end{cases}
\end{equation}
By the Galilean invariance of the system, without loss of generality, we may assume either $u_-=0$ or $u_+=0$.

We define the relative mechanical energy with respect to $(\bar{\rho}(x),\bar{m}(x))=(\bar{\rho}(x), \bar{\rho}(x)\bar{u}(x))$:
\beqa\label{relative-energy}
\overline{\eta^*}(\rho,m)=&\,\eta^*(\rho,m)-\eta^*(\bar{\rho},\bar{m})-\nabla\eta^*(\bar{\rho},\bar{m})\cdot(\rho-\bar{\rho},m-\bar{m})\\
  =&\,\frac{1}{2}\rho|u-\bar{u}|^2+\overline{h}(\rho,\bar{\rho})\geq0,
\eeqa
where $\overline{h}(\rho,\bar{\rho})=h(\rho)-h(\bar{\rho})-h'(\bar{\rho})(\rho-\bar{\rho})$,
and $h(\rho)=\frac{\kappa \rho^\gamma}{\gamma-1}$.

\begin{definition}
\, The {\it relative total mechanical energy} \rm for \eqref{transonic} with respect to the end-states $(\rho_\pm,m_\pm)$ through
 $(\bar{\rho},\bar{m})$ is
 \begin{equation*}
  E[\rho,m](t):=\int_{-\infty}^\infty\overline{\eta^*}(\rho(t,x), m(t,x))A(x)\, dx\geq 0.
 \end{equation*}
 We say that a pair $(\rho,m)$ is of {\it relative finite-energy} with respect to the end-states $(\rho_\pm,m_\pm)$
 if $E[\rho,m](t)<\infty$ for any $t>0$.
\end{definition}

\subsection{\, Main Theorems}
We now present the two main theorems.

\begin{theorem}[Transonic Nozzle Solutions]\label{thm:main}
\, Let the cross-sectional area function $A(x)>0$ be a $C^2$--function satisfying \eqref{1.3a} or \eqref{1.3b}.
Assume that $(\rho_0,m_0)\in L^1_{loc}(\mathbb{R})$ with $\rho_0(x)\geq0$ is of relative finite-energy
with respect to the prescribed end-states $(\rho_\pm,m_\pm)$.
Then there exists a
global finite-energy entropy solution $(\rho,m)(t,x)$ of the transonic nozzle problem
\eqref{transonic}--\eqref{prob:Cauchy}
in the sense of Definitions {\rm \ref{def:entropysol}}--{\rm 2.2} such that
\begin{equation*}
(\rho,m)\in L^p_{loc}(\mathbb{R}^2_+)\times L^q_{loc}(\mathbb{R}^2_+) \qquad\,\,
\mbox{for $p\in[1,\gamma+1)$ and $q\in[1,\frac{3(\gamma+1)}{\gamma+3})$}.
\end{equation*}
\end{theorem}

\begin{remark}
 \, Condition \eqref{1.3a} or \eqref{1.3b} allows for the nozzles to open up at one end {\rm (}with an unbounded cross-sectional area{\rm )},
and to shrink at both ends with rate $|x|^{-\al}$ for any $\al>0$ {\rm (}closed ends{\rm )}. In particular,
a positive lower bound, or upper bound, of the cross-sectional area function is not required.
\end{remark}

The existence of entropy solutions presented in the above theorem will be established
as a vanishing viscosity limit
of appropriately designed approximate solutions.
We will therefore first develop the viscosity method to construct such approximate solutions and
then derive the uniform estimates of the approximate solutions so that they satisfy the requirements of the $L^p$ compensated compactness
framework in \cite{ChenPerep1}.
Moreover, as a consequence of the approach and techniques developed below,
we are able to prove the following result for the spherically symmetric Euler equations,
including the unsolved case: $\gamma>3$.

\begin{theorem}\label{thm:sph-symm}
\, Let $(\rho_0,m_0)\in (L^1_{loc}(\R_+))^2$ be finite-energy initial data such that $\rho_0(x)\geq0$ and
$$
E_*[\rho_0,m_0]:=\int_0^\infty\eta^*(\rho_0,m_0)x^{n-1}dx <\infty.
$$
Then, for the whole range $\gamma\in (1, \infty)$, there exists a global entropy solution $(\rho, m)(t,x)$
of the spherically symmetric Euler equations
with $A(x)=\omega_n x^{n-1}, x\ge0$, where $\omega_n$ is the surface area of the unit sphere in $\R^n$, such that
$(\rho,m)\in L^p_{loc}(\mathbb{R}^2_+)\times L^q_{loc}(\mathbb{R}^2_+)$ for $p\in[1,\gamma+1)$ and $q\in[1,\frac{3(\gamma+1)}{\gamma+3})$, and
\begin{equation*}
E_*[\rho,m](t):=\int_0^\infty\eta^*(\rho,m)x^{n-1}dx\leq E_*[\rho_0,m_0]<\infty.
\end{equation*}
\end{theorem}

\medskip
\section{\, Viscosity Approximate Solutions}\label{sec:artificial}
We now develop the viscosity method to construct the viscosity approximate solutions
of the Cauchy problem \eqref{transonic}--\eqref{prob:Cauchy} and make some necessary estimates.
Throughout this section and \S \ref{sec:uniform}--\S 5,
we first assume that
there exist $A_0,A_1, A_2>0$ such that conditions \eqref{1.4}--\eqref{1.5}
hold for the cross-sectional area function $A(x)$ for simplicity of presentation.

For fixed $\varepsilon>0$, we consider the following approximate equations for $(t,x)\in \mathbb{R}_+\times(a,b)$:
\begin{equation}\label{approx}
 \begin{cases}
  \rho^\varepsilon_t+m^\varepsilon_x+\frac{A'(x)}{A(x)}m^\varepsilon=\varepsilon\big(\rho^\varepsilon_{xx}+\frac{A'(x)}{A(x)}\rho^\varepsilon_x\big),\\[1mm]
  m^\varepsilon_t+\big(\frac{(m^\varepsilon)^2}{\rho^\varepsilon} +p_\delta(\rho^\varepsilon)\big)_x
   +\frac{A'(x)}{A(x)}\frac{(m^\varepsilon)^2}{\rho^\varepsilon}=\varepsilon\big(m^\varepsilon_x+\frac{A'(x)}{A(x)}m^\varepsilon\big)_x,
 \end{cases}
\end{equation}
 where
\begin{equation}\label{3.2a}
p_\delta(\rho)=\kappa\rho^\gamma+\delta\rho^2, \hspace{6mm}\,\, \mbox{$\delta=\delta(\varepsilon)>0$ with $\delta(\varepsilon)\rightarrow0$ as $\varepsilon\rightarrow 0$}.
\end{equation}

The introduction of this additional term $\delta \rho^2$ into
the pressure function prevents cavitation ({\it i.e.} the formation of a vacuum state)
in the approximate solutions when $\varepsilon>0$.
Here $a=a(\varepsilon)$ and $b=b(\varepsilon)$ are chosen such that
\begin{align}
&a(\varepsilon)<-L_0,   \hspace{6mm} a(\varepsilon)\rightarrow-\infty\,\, \text{ as $\varepsilon\rightarrow0$,}\label{a}\\
&b(\varepsilon)>L_0,    \hspace{9.5mm} b(\varepsilon)\rightarrow\infty\,\, \text{ as $\varepsilon\rightarrow0$}\label{b}
\end{align}
for $L_0$ in \eqref{2.5}, independent of $\varepsilon$.
We can take $a(\varepsilon)$ and $b(\varepsilon)$ to diverge at the same rate in $\varepsilon\to 0$,
but this rate is not arbitrary.
However, we have the freedom to design the approximating problems in a convenient manner
so that  this rate can be chosen carefully in \S \ref{sec:uniform}--\S 5.

We pose approximate initial and Dirichlet boundary data as follows:
\begin{equation}\label{BCs}
 \begin{cases}
  (\rho,m)|_{t=0}=(\rho_0^\varepsilon,m_0^\varepsilon)(x), \quad &x\in(a,b),\\
  (\rho,m)|_{x=a}=(\rho^\varepsilon_-,m^\varepsilon_-), \quad &t\geq 0,\\
  (\rho,m)|_{x=b}=(\rho^\varepsilon_+,m^\varepsilon_+), \quad &t\geq 0,
 \end{cases}
\end{equation}
where $\rho^\varepsilon_\pm>0$, and $(\rho^\varepsilon_\pm,m^\varepsilon_\pm)\rightarrow(\rho_\pm,m_\pm)$
as $\varepsilon\rightarrow0$.
The imposition that the end-states for $\rho$ are strictly positive is to prevent
the possibility of cavitation.
One of the motivations for us to impose the Dirichlet boundary data for the approximate solutions
is to allow for the use of the weak maximum principle for obtaining the $L^\infty$ estimate for the approximate solutions
for the whole range $\gamma\in (1,\infty)$;
see Lemma \ref{lemma:max-principle}.

With this carefully designed initial-boundary value problem \eqref{approx}--\eqref{BCs},
the next goal of this section is now to establish the existence of globally defined approximate solutions.
Throughout this section, for $\beta\in(0,1)$, $C^{2+\beta}([a,b])$ and $C^{2+\beta,1+\frac{\beta}{2}}(Q_T)$ denote
the usual H\"{o}lder and parabolic H\"{o}lder spaces on the interval $[a,b]$ and the parabolic cylinder $Q_T:=[0,T]\times[a,b]$,
respectively, as defined in \cite{LadySoloUral}.

\begin{theorem}\label{artsol}
\, For $\varepsilon>0$, let $(\rho_0^\varepsilon,m_0^\varepsilon)\in\big(C^{2+\beta}([a,b])\big)^2$ be a sequence
of functions such that{\rm :}
\begin{enumerate}
\item[\rm (i)] $\inf_{a\leq x\leq b}\rho_0^\varepsilon(x)>0${\rm ;}

\smallskip
\item[\rm (ii)] $(\rho_0^\varepsilon,m_0^\varepsilon)$ satisfies \eqref{BCs} and the compatibility conditions at $x=a(\varepsilon)$ and $b(\varepsilon)${\rm :}
   \begin{equation*}
    \big(A(x) m_{0}^\varepsilon\big)_x =\varepsilon \big(A(x)\rho_{0,x}^\varepsilon\big)_x,
    \hspace{2.6mm}\big(A(x)\frac{(m_0^\varepsilon)^2}{\rho_0^\varepsilon}\big)_x +A(x) p_\delta(\rho_0^\varepsilon)_x=\varepsilon \big(A(x)m_0^\varepsilon\big)_x;
   \end{equation*}
\item[\rm (iii)] $\int_a^b\big(\frac{(m_0^\varepsilon)^2}{2\rho_0^\varepsilon}+\frac{\kappa(\rho_0^\varepsilon)^\gamma}{\gamma-1}\big)A(x)\,dx<\infty${\rm ;}

\smallskip
\item[\rm (iv)] $(\delta(\varepsilon),a(\varepsilon),b(\varepsilon))$ satisfy \eqref{3.2a}--\eqref{b}.
\end{enumerate}
Then there exists a unique global solution $(\rho^\varepsilon,m^\varepsilon)(t,x)$ of problem \eqref{approx}--\eqref{BCs} for
$\gamma\in(1,\infty)$ such that $(\rho^\varepsilon,m^\varepsilon)\in\big(C^{2+\beta,1+\frac{\beta}{2}}(Q_T)\big)^2$ with
$\inf_{Q_T}\rho^\varepsilon(t,x)>0$ for all $T>0$.
\end{theorem}

We assume from now on
that $a=a(\varepsilon)$ and $b=b(\varepsilon)$ are constants depending on $\varepsilon>0$ such that
\begin{equation}\label{3.6}
\big(1+\big|\big(\frac{A'}{A}\big)'\big|\big)\varepsilon|a-b|\leq M,
\end{equation}
where $M>0$ is a constant, independent of $\varepsilon$.
For simplicity of presentation, we drop the explicit $\varepsilon$--dependence of all functions in this section, since
the results hold for each fixed $\varepsilon>0$.
Moreover, we use the following notation:
\begin{equation}\label{3.7}
h_\delta(\rho)=\rho e_\delta(\rho),\hspace{5mm} e_\delta(\rho)=\int_0^\rho\frac{p_\delta(s)}{s^2}\,ds.
\end{equation}

By the theory detailed in \cite{LadySoloUral}, in order to establish Theorem \ref{artsol},
it suffices to show that, for a
generic $C^{2,1}(Q_T)$--solution of \eqref{approx},
the following \it a priori \rm bound holds:
$$
\big\|\big(\rho, \rho^{-1},\frac{m}{\rho}\big)\big\|_{L^\infty(Q_T)}<\infty,
$$
where the bound may depend on $\varepsilon$, $T$, and the initial-boundary data.
Then this implies that
a generic solution of the equations with the prescribed regularity remains in a bounded region,
away from the singularities of the coefficients of the quasilinear parabolic system \eqref{approx}.
The results in \S 5 and Theorem 7.1 in \S 7 of \cite{LadySoloUral}
then lead to the conclusion of Theorem 3.1.

\smallskip
The remainder of this section consists
of  several energy estimates in order to deduce these bounds.

Our fundamental energy estimate concerns the control of the relative mechanical energy
for system \eqref{approx}.
We define a modified entropy pair:
\begin{equation*}
 \eta_\delta^*=\frac{m^2}{2\rho}+h_\delta(\rho),\hspace{6mm} q_\delta^*=\frac{m^3}{2\rho^2}+mh_\delta'(\rho),
\end{equation*}
and further  modify these to obtain
\begin{equation*}
\begin{split}
\overline{\eta_\delta^*}(\rho,m)
=\,&\eta_\delta^*(\rho,m)-\eta_\delta^*(\bar{\rho},\bar{m})-\nabla\eta_\delta^*(\bar{\rho},\bar{m})\cdot(\rho-\bar{\rho},m-\bar{m})\\
=\,&\frac{1}{2}\rho|u-\bar{u}|^2+\overline{h_\delta}(\rho,\bar{\rho})\geq0,
 \end{split}
\end{equation*}
where  $\overline{h_\delta}(\rho,\bar{\rho})=h_\delta(\rho)-h_\delta(\bar{\rho})-h_\delta'(\bar{\rho})(\rho-\bar{\rho})$.

\smallskip
Then the \it total relative mechanical energy \rm for \eqref{approx} with respect to the end-states $(\rho_\pm,m_\pm)$ through
$(\bar{\rho},\bar{m})$ is
\begin{equation*}
  E[\rho,m](t):=\int_a^b\overline{\eta_\delta^*}(\rho(t, x), m(t,x))A(x)\,dx\geq 0.
 \end{equation*}

\medskip
Now we obtain our main energy estimate.
\begin{proposition}\label{energy}
 Let $$E_0:=\sup_{\varepsilon>0}\int_a^b\overline{\eta_\delta^*}(\rho_0^\varepsilon(x),m_0^\varepsilon(x))A(x)\,dx<\infty.$$
 Then there exists $M>0$, independent of $\varepsilon$, such that, for all $\varepsilon>0$,
 \begin{equation}\label{energy1}
  \begin{split}
   &\sup_{t\in[0,T]}\int_a^b\big(\frac{1}{2}\rho |u-\bar{u}|^2 + \overline{h_\delta}(\rho,\bar{\rho})\big)A(x)\,dx\\
   &+\varepsilon\int_{Q_T}\Big(h_\delta''(\rho)\rho_x^2+\rho u_x^2+\big|\big(\frac{A'}{A}\big)'\rho u(u-\bar{u})\big|\Big)A(x)\,dx\,dt\leq M(E_0+1).
  \end{split}
 \end{equation}
Furthermore, for any $t\in[0,T]$, $|\{\rho(t,\cdot)>\frac{3}{2}\bar{\rho}\}|\leq c_2 E_0$, where $c_2>0$
may depend on $T$.
\end{proposition}

Before we prove Proposition \ref{energy},
we note that there exists a constant $c_1>0$, independent of $\varepsilon$,
depending only on $\bar{\rho}$ and $\gamma$ such
that $\overline{h_\delta}(\rho,\bar{\rho})\geq c_1\rho(\rho^\theta-\bar{\rho}^\theta)^2$.

\begin{proof} \,
Multiplying the first equation in \eqref{approx} by $\big(\overline{\eta_\delta^*}\big)_\rho A(x)$
and
the second equation by $\big(\overline{\eta_\delta^*}\big)_m A(x)$, and then adding them together,
we have
\begin{align}
&\big(\overline{\eta_\delta^*}A(x)\big)_t
  +\big(\big(q_\delta^*-m(\eta_\delta^*)_\rho(\bar{\rho},\bar{m})-(\tfrac{m^2}{\rho}+p_\delta(\rho))(\eta_\delta^*)_m(\bar{\rho},\bar{m})\big)A(x)\big)_x\nonumber\\
  &\,\,+\big((\eta_\delta^*)_\rho(\bar{\rho},\bar{m})\big)_xmA(x)+\big((\eta_\delta^*)_m(\bar{\rho},\bar{m})\big)_x\big(\tfrac{m^2}{\rho}+p_\delta(\rho)\big)A(x)\nonumber\\[1mm]
  &\,\,+p_\delta(\rho)(\eta_\delta^*)_m(\bar{\rho},\bar{m})A'(x)\nonumber\\
&=\varepsilon(A(x)\rho_x)_x(\overline{\eta_\delta^*})_\rho +\varepsilon A(x)\big(A(x)^{-1}(A(x)m)_x\big)_x(\overline{\eta_\delta^*})_m. \label{eq:star}
\end{align}
Observe that, at the end-points $a$ and $b$,
$$
q_\delta^*-m(\eta_\delta^*)_\rho(\bar{\rho},\bar{m})-\big(\frac{m^2}{\rho}+p_\delta(\rho)\big)\big(\eta_\delta^*\big)_m(\bar{\rho},\bar{m})
=-\frac{m_\pm}{\rho_\pm}p_\delta(\rho_\pm).
$$
Then
\begin{align*}
&\frac{dE}{dt}-\frac{m}{\rho}p_\delta(\rho)\vert_a^b\\
&+\int_a^b\Big(m\big((\eta_\delta^*)_\rho(\bar{\rho},\bar{m})\big)_x
   +\big((\eta_\delta^*)_m(\bar{\rho},\bar{m})\big)_x\big(\tfrac{m^2}{\rho}+p_\delta(\rho)\big)\\
  &\qquad\quad\,
   +\frac{A'(x)}{A(x)}(\eta_\delta^*)_m(\bar{\rho},\bar{m})p_\delta(\rho)-\varepsilon\big(\frac{A'(x)}{A(x)}\big)'m\big((\eta_\delta^*)_m-(\eta_\delta^*)_m(\bar{\rho},\bar{m})\big)\Big)A(x)\,dx\\
&=-\varepsilon\int_a^b (\rho_x,m_x)\nabla^2\overline{\eta_\delta^*}\,(\rho_x,m_x)^\top A(x)\,dx\\
  &\quad +\varepsilon\int_a^b \big(\rho_x((\eta^*_\delta)_\rho(\bar{\rho},\bar{m}))_x+m_x((\eta^*_\delta)_m(\bar{\rho},\bar{m}))_x\big)A(x)\,dx.
 \end{align*}
Note that $\rho+p_\de(\rho)\leq M\big(\overline{h_\de}(\rho,\bar\rho)+1\big)$ and
$\|A'\|_{L^1(-\infty,L_0)}\leq M$ and that $(\bar\rho_x, \bar m_x)$ are supported in $(-L_0,L_0)$.
Moreover, since $\bar m=0$ and, for $x>L_0$, $(\eta_\delta^*)_m(\bar{\rho},\bar{m})=0$  (recall we have used the
Galilean invariance to impose that $u_+=0$ without loss of generality), we have
\beqas
&\int_a^b\Big( m\big((\eta_\delta^*)_\rho(\bar{\rho},\bar{m})\big)_x+\big((\eta_\delta^*)_m(\bar{\rho},\bar{m})\big)_x\big(\tfrac{m^2}{\rho}+p_\delta(\rho)\big)\\
  &\qquad+\frac{A'(x)}{A(x)}(\eta_\delta^*)_m(\bar{\rho},\bar{m})p_\delta(\rho)\Big)A(x)\,dx\\
&\leq M(E+1).
\eeqas
 Moreover, integrating by parts gives
 \begin{align*}
  \varepsilon&\int_a^b \big(\rho_x((\eta^*_\delta)_\rho(\bar{\rho},\bar{m}))_x+m_x((\eta^*_\delta)_m(\bar{\rho},\bar{m}))_x\big)A(x)\,dx\\
&\leq M\varepsilon\big(E+L_0+\|A'\|_{L^1(-L_0,L_0)}\big).
 \end{align*}
 Finally, we have
 \begin{align*}
  \int_a^b&\varepsilon\Big|\big(\frac{A'}{A}\big)'\,m\big((\eta_\delta^*)_m-(\eta_\delta^*)_m(\bar{\rho},\bar{m})\big)\Big|A(x)\,dx\\
  &\leq \varepsilon ME+\eps M|b-a|\big(1+\big|\big(\frac{A'}{A}\big)'\big|\big).
 \end{align*}

Since $(\rho_x,m_x)\nabla^2\overline{\eta_\delta^*}(\rho_x,m_x)^\top$ dominates $h_\delta''(\rho)|\rho_x|^2+\rho|u_x|^2$ as in \cite{ChenPerep2},
we may combine all of these inequalities to obtain
\begin{align*}
  \frac{dE}{dt}+\varepsilon\int_a^b&\big((2\delta+\kappa\gamma\rho^{\gamma-2})|\rho_x|^2+\rho|u_x|^2\big)A(x)\,dx\leq M(E+1).
\end{align*}
 Hence, by the Gronwall inequality, we have
 \begin{equation}\label{++}
  E(T)+\varepsilon\int_{Q_T}\big((2\delta+\kappa\gamma\rho^{\gamma-2})|\rho_x|^2+\rho|u_x|^2\big)A(x)\,dx\leq M(T)(E_0+1).
 \end{equation}
 In particular, we obtain
 \begin{equation}\label{eta*}
  \sup_{t\in[0,T]}\int_a^b\big(\rho(u-\bar{u})^2+\overline{h_\delta}(\rho,\bar{\rho})\big)A(x)\, dx\leq M(E_0+1).
 \end{equation}

 Since $A(x)\geq A_0>0$, and $\overline{h_\delta}(\rho,\bar{\rho})\geq 0$ is quadratic in $(\rho-\bar{\rho})$
 for $\rho$ near $\bar{\rho}$ and grows as $\rho^{\max\{\gamma,2\}}$ for large $\rho$, we conclude
 \begin{equation*}
  |\{\rho(t,\cdot)>\frac{3}{2}\check{\rho}\}|\leq c E_0,
 \end{equation*}
 where $\check{\rho}=\max\{\rho_-,\rho_+\}$.
\qed
\end{proof}

The following lemma is a simple consequence of the fundamental theorem of calculus
and the main energy estimate (Proposition \ref{energy}),
so its proof is omitted; see the proof of \cite[Lemma 2.1]{ChenPerep2}.

\begin{lemma}\label{energy2}
 There exists $C=C(\varepsilon,T,E_0)>0$ such that
 \begin{equation*}
  \int_0^T\|\rho(t,\cdot)\|_{L^\infty(a,b)}^{2\max\{2,\gamma\}}dt\leq C.
 \end{equation*}
\end{lemma}
From now on, the constants labelled $C$ may depend on $\varepsilon$, whereas the constants labelled $M$
are independent of $\varepsilon$.
The following maximum principle result for the Riemann invariants
gives \textit{a priori} control on $\|(\rho,u)\|_{L^\infty}$
in terms of the Dirichlet boundary data and initial data.

\begin{lemma}\label{lemma:max-principle}
 There exists $C=C(\varepsilon,T,E_0)$ such that, for any $t\in[0,T]$,
 \begin{equation*}
  \|(\rho,u)\|_{L^\infty(Q_T)}\leq C\big(\|u_0+R(\rho_0)\|_{L^\infty(a,b)}+\|u_0-R(\rho_0)\|_{L^\infty(a,b)}\big)+C(\rho_\pm,u_\pm),
 \end{equation*}
 where
 $$
 R(\rho):=\int_0^\rho\frac{\sqrt{p_\delta'(s)}}{s}\,ds.
 $$
\end{lemma}

\begin{proof}
 \, The eigenvalues of the Euler system with pressure function $p_\delta(\rho)$ are
$$
\lambda_1:=u-\sqrt{p_\delta'(\rho)},\hspace{6mm}\lambda_2:=u+\sqrt{p_\delta'(\rho)},
$$
and the corresponding Riemann invariants are
$$
w:=u+R(\rho),\hspace{6mm}z:=u-R(\rho).
$$
We first note that these Riemann invariants are quasiconvex:
$$
\nabla^\perp w\nabla^2w(\nabla^\perp w)^\top\geq0,\hspace{6mm}
-\nabla^\perp z\nabla^2z(\nabla^\perp z)^\top\geq0.
$$
Following \cite[Lemma 2.2]{ChenPerep2}, we multiply the first equation in \eqref{approx} by $w_\rho$,
multiply the second by $w_m$, and then add them together
with a short calculation to derive that
 \begin{align*}
 w_t+\big(\lambda_2-\varepsilon\frac{A'}{A}\big)w_x-\varepsilon w_{xx}
 +\varepsilon(\rho_x,m_x)\nabla^2w(\rho_x,m_x)^\top=\varepsilon \big(\frac{A'}{A}\big)' u-u\sqrt{p_\delta'}\frac{A'}{A}.
 \end{align*}

Notice that
$$
(\rho_x ,m_x)=\alpha\nabla w+\beta\nabla^\perp w,
$$
where
$$
\alpha=\frac{w_x}{|\nabla w|^2},\hspace{6mm}\beta=\frac{\rho_xw_m-m_xw_\rho}{|\nabla w|^2}.
$$
With this notation, we re-write the above as
\begin{align*}
 w_t+\lambda w_x-\varepsilon w_{xx}
 =-\varepsilon\beta^2\nabla^\perp w\nabla^2w(\nabla^\perp w)^\top
   -u\sqrt{p_\delta'}\frac{A'}{A}+\varepsilon \big(\frac{A'}{A}\big)' u,
\end{align*}
where
$$
\lambda:=\lambda_2-\varepsilon\frac{A'}{A}+\varepsilon\alpha\frac{\nabla w\nabla^2w(\nabla w)^\top}{|\nabla w|^2}
 +2\varepsilon\beta\frac{\nabla^\perp w\nabla^2 w(\nabla w)^\top}{|\nabla w|^2}.
$$
Then, using the quasiconvexity property of $w$, we find that
$$
\tilde{w}_t+\lambda\tilde{w}_x-\varepsilon \tilde{w}_{xx}\leq0
$$
for $\tilde{w}$ defined by
$$
\tilde{w}(t,x):=w(t,x)-\int_0^t\big\|u\sqrt{p_\delta'}\frac{A'}{A}-\varepsilon\big(\frac{A'}{A}\big)'u\big\|_{L^\infty(a,b)}\,d\tau.
$$
The maximum principle for parabolic equations gives us
$$
\max_{Q_t}\tilde{w}\leq\max\{w_0,\max_{[0,t]\times(\{a\}\cup\{b\})}\tilde{w}\}.
$$
Now we have
\begin{align*}
&\max_{[0,t]\times(\{a\}\cup\{b\})}\tilde{w}\\
&\leq\,\max\{u_-+R(\rho_-),u_++R(\rho_+)\}\\
&\quad +\check{\rho}\sup_{(a,b)}\big\{\big|\frac{A'}{A}\big|,\big|\big(\frac{A'}{A}\big)'\big|\big\}
\int_0^t\big(1+\|\rho(\tau,\cdot)\|_{L^\infty(a,b)}^{\frac{1}{2}\max\{1,\gamma-1\}}\big)\|u(\tau,\cdot)\|_{L^\infty(a,b)}\,d\tau.
\end{align*}
Thus, we obtain
\begin{align*}
 \max_{Q_t}w\leq &\,\max_{(a,b)}w_0+\max\{u_-+R(\rho_-),u_++R(\rho_+)\}\\
 &+C\int_0^t\big(1+\|\rho(\tau,\cdot)\|_{L^\infty(a,b)}^{\frac{1}{2}\max\{1,\gamma-1\}}\big)\|u(\tau,\cdot)\|_{L^\infty(a,b)}\,d\tau.
\end{align*}
Similarly,
\begin{align*}
 \max_{Q_t}(-z)\leq &\, \max_{(a,b)}(-z_0)+\max\{u_--R(\rho_-),u_+-R(\rho_+)\}\\
 &+C\int_0^t\big(1+\|\rho(\tau,\cdot)\|_{L^\infty(a,b)}^{\frac{1}{2}\max\{1,\gamma-1\}}\big)\|u(\tau,\cdot)\|_{L^\infty(a,b)}\,d\tau.
\end{align*}
Noting that $w_0$ satisfies the corresponding boundary condition to \eqref{BCs} and $\rho\geq0$, we see that
\begin{align*}
 &\max_{Q_t}|u|\\
 &\,\,\leq \,\|w_0\|_{L^\infty}+\|z_0\|_{L^\infty}
 +C\int_0^t\big(1+\|\rho(\tau,\cdot)\|_{L^\infty(a,b)}^{\frac{1}{2}\max\{1,\gamma-1\}}\big)\|u(\tau,\cdot)\|_{L^\infty(a,b)}\,d\tau.
\end{align*}
Since $\max\{1,\gamma-1\}<2\gamma$, we obtain that, by the H\"{o}lder inequality and Lemma 3.1,
\begin{align*}
\max_{Q_t}|u|^2\leq&\, C\big(\max_{(a,b)}|w_0|+\max_{(a,b)}|z_0|\big)^2\\
& +C\Big(\int_0^t\big(1+\|\rho(\tau,\cdot)\|_{L^\infty(a,b)}^{\max\{1,\gamma-1\}}\big)\,d\tau\Big)
  \Big(\int_0^t\|u(\tau,\cdot)\|_{L^\infty(a,b)}^2\,d\tau\Big)\\
\leq &\, C\big(\max_{(a,b)}|w_0|+\max_{(a,b)}|z_0|\big)^2 +C\int_0^t\|u(\tau,\cdot)\|_{L^\infty(a,b)}^2\,d\tau.
\end{align*}
We then conclude the estimate for $\|u\|_{L^\infty(Q_t)}$ by using the Gronwall inequality.

The estimate for $\rho$ is now a direct consequence of the estimates of the Riemann invariants
and the estimate of $u$ above. In particular, we have
\beqas
 \max_{Q_t}R(\rho)=\tfrac{1}{2}\max_{Q_t}(w-z)\leq C+C\int_0^t\big(1+\|\rho(\tau,\cdot)\|_{L^\infty(a,b)}^{\tfrac{1}{2}\max\{1,\gamma-1\}}\big)\,d\tau\leq C.\nonumber
\eeqas
\qed\end{proof}

The next result provides the afore-mentioned higher order energy estimates
that we require to conclude the upper bound
of $\|\rho^{-1}\|_{L^\infty(a,b)}$.

\begin{lemma}\label{energy4}
	There exists $C(\varepsilon,\|(\rho_0,u_0)\|_{L^\infty(a,b)},\|(\rho_0,m_0)\|_{H^1(a,b)}, T,\gamma)>0$ such that
	\begin{equation*}
	\sup_{t\in[0,T]}\int_a^b\big(|\rho_x|^2+|m_x|^2\big) dx+\int_{Q_T}\big(|\rho_{xx}|^2+|m_{xx}|^2\big)\, dx\,dt\leq C.
	\end{equation*}
\end{lemma}

\begin{proof}
 \, Throughout this proof,
we make frequent use of the upper bounds on $(\rho, u)$ from Lemma \ref{lemma:max-principle}.

Multiplying the first equation in \eqref{approx} by $\rho_{xx}$ and the second by $m_{xx}$,
we obtain
\begin{align*}
&(\rho_t\rho_x)_x -\frac{1}{2}\big(|\rho_x|^2\big)_t
 +(m_tm_x)_x-\frac{1}{2}\big(|m_x|^2\big)_t-\varepsilon\big(|\rho_{xx}|^2+|m_{xx}|^2\big)\\
&= -m_x\rho_{xx}-(\rho u^2+p_\delta)_xm_{xx}-\frac{A'}{A}m\rho_{xx}-\frac{A'}{A}\big(\rho u^2m_{xx}-\varepsilon\rho_x\rho_{xx}\big)\\
&\quad +\varepsilon\big(\frac{A'}{A}m\big)_xm_{xx}.
\end{align*}
Integrating over $Q_T$ and recalling that $\rho$ and $m$ are constant on $x=a$ and $b$, we have
\begin{align*}
\frac{1}{2}\int_a^b& \big(|\rho_x|^2+|m_x|^2\big)\big\vert_0^T\,dx+\varepsilon\int_{Q_T}\big(|\rho_{xx}|^2+|m_{xx}|^2\big)\,dx\,dt\\
=&\int_{Q_T}\big(m_x\rho_{xx}+\frac{A'}{A}m\rho_{xx}\big)\,dx\,dt+\int_{Q_T}(\rho u^2 +p_\delta)_xm_{xx}\,dx\,dt\\
 &+\int_{Q_T}\frac{A'}{A}\big(\rho u^2m_{xx}-\varepsilon\rho_x\rho_{xx}\big)\,dx\,dt-\varepsilon\int_{Q_T}\big(\frac{A'}{A}m\big)_xm_{xx}\,dx\,dt.
\end{align*}
Thus, using the uniform bounds from the maximum principle, we have
\begin{align*}
 \int_a^b&\big(|\rho_x(T,x)|^2+|m_x(T,x)|^2\big)\,dx+\varepsilon\int_{Q_T}\big(|\rho_{xx}|^2+|m_{xx}|^2\big)\,dx\,dt\\
 \leq&\,\Delta\int_{Q_T}\big(|\rho_{xx}|^2+|m_{xx}|^2\big)\,dx\,dt+\int_a^b\big(|\rho_{0,x}(x)|^2+|m_{0,x}(x)|^2\big)\,dx\\
 &+C_\Delta \int_{Q_T}\big(|\rho_x|^2+|m_x|^2\big)\,dx\,dt+C_\Delta,
\end{align*}
where $\Delta>0$.
Absorbing the first term into the left by taking $\De$ small, we conclude the expected estimate.
\qed\end{proof}

It now remains only to demonstrate an \textit{a priori} upper bound for $\rho^{-1}$.
To this end, we define
\begin{equation*}
 \phi(\rho)=\begin{cases}
             \frac{1}{\rho}-\frac{1}{\tilde{\rho}}+\frac{\rho-\tilde{\rho}}{\tilde{\rho}^2}, & \rho<\tilde{\rho},\\
             0, &\rho>\tilde{\rho}
            \end{cases}
\end{equation*}
for some $\tilde{\rho}>0$. Note that
$$
\phi'(\rho)=\big(-\frac{1}{\rho^2}+\frac{1}{\tilde{\rho}^2}\big)\chi_{\{\rho<\tilde{\rho}\}},
\qquad \phi''(\rho)=\frac{2}{\rho^3}\chi_{\{\rho<\tilde{\rho}\}}.
$$

\begin{lemma}\label{energy6}
 There exists $C>0$ depending on $\|\phi(\rho_0)\|_{L^1(a,b)}$ and the other parameters of the problem such that
 $$\sup_{t\in[0,T]}\int_a^b\phi(\rho(t,\cdot))\,dx+\int_{Q_T}\frac{|\rho_x|^2}{\rho^3}\,dx\,dt\leq C.$$
\end{lemma}

\begin{proof}
 \, Multiplying the first equation of \eqref{approx} by $\phi'(\rho)$, we have
 \begin{align*}
 &\phi_t+(u\phi)_x-\varepsilon\phi_{xx}+2\varepsilon\frac{\rho_x^2}{\rho^3}\chi_{\{\rho<\tilde{\rho}\}}\\
  &=2\big(\frac{1}{\rho}-\frac{1}{\tilde{\rho}}\big)u_x\chi_{\{\rho<\tilde{\rho}\}}+\frac{A'}{A}\rho u\big(\frac{1}{\rho^2}-\frac{1}{\tilde{\rho}^2}\big)\chi_{\{\rho<\tilde{\rho}\}}+\varepsilon\frac{A'}{A}\rho_x\big(\frac{1}{\tilde{\rho}^2}-\frac{1}{\rho^2}\big)\chi_{\{\rho<\tilde{\rho}\}}.
 \end{align*}
 Integrating in $(t,x)$ and using the boundary conditions with $\tilde{\rho}<\inf_{(a,b)}\rho_0(x)$ chosen,
 we obtain
 \begin{align*}
  \int_a^b&\phi(\rho)\,dx+\varepsilon\int_{Q_T\cap\{\rho<\tilde{\rho}\}}\frac{|\rho_x|^2}{\rho^3}\,dx\,dt\\
  \leq&\, \Big|\int_{Q_T\cap\{\rho<\tilde{\rho}\}}2\big(\frac{1}{\rho}-\frac{1}{\tilde{\rho}}\big)u_x\,dx\,dt\Big|
    +\Big|\int_{Q_T\cap\{\rho<\tilde{\rho}\}}\frac{A'}{A}\rho u\big(\frac{1}{\rho^2}-\frac{1}{\tilde{\rho}^2}\big)\,dx\,dt\Big|\\
  &+\Big|\int_{Q_T\cap\{\rho<\tilde{\rho}\}}\varepsilon\frac{A'}{A}\rho_x\big(\frac{1}{\tilde{\rho}^2}-\frac{1}{\rho^2}\big)\,dx\,dt\Big|\\
  \leq&\, \De\int_{Q_T\cap\{\rho<\tilde{\rho}\}}\frac{|\rho_x|^2}{\rho^3}\,dx\,dt+C_\De\Big(1+\int_{Q_T}\phi(\rho)\,dx\,dt\Big)
 \end{align*}
 and conclude the estimate by the Gronwall inequality, where we have integrated by parts
 and applied the Young inequality in the first term after the first inequality above.
\qed\end{proof}

For $x\in(a,b)$, we take $x_0$ to be the closest point to $x$
such that $\rho(t,x_0)=\tilde{\rho}$, where $\tilde{\rho}=\min_{x\in (a,b)}\rho_0(x)$;
otherwise,
we already have an upper bound on $\rho^{-1}$.
By writing
$$
\big|\rho(t,x)^{-1}\big|
\leq\Big|\int_{x_0}^x\frac{\rho_y(t,y)}{\rho^2(t,y)}\,dy\Big|+\big|(\tilde{\rho})^{-1}\big|,
$$
we obtain
\begin{align}\label{3.2}
 \int_0^T\big\|\rho(t,\cdot)^{-1}\big\|_{L^\infty(a,b)}\,dt
 \leq&\, C+C\Big(\int_{Q_T}\frac{|\rho_x|^2}{\rho^3}\,dx\,dt\Big)^{\frac{1}{2}}\Big(\big(\int_{Q_T}\phi(\rho)\,dx\,dt\big)^{\frac{1}{2}}+1\Big)\nonumber\\
 \leq&\, C\bigg(1+\Big(\int_{Q_T}\frac{|\rho_x|^2}{\rho^3}\, dxdt\Big)^{\frac{1}{2}}\bigg).
\end{align}

Finally, we use the previous lemmas to show the lower bound on $\rho$.

\begin{lemma}\label{energy7}
There exists $C>0$, depending on $\eps, T$, and the other parameters of the problem, such that
\begin{eqnarray*}
 C^{-1}\leq\rho(t,x)\leq C,\qquad\,
 \int_0^T\big\|\big(\frac{m_x}{\rho},\frac{\rho_x}{\rho},u_x\big)(t,\cdot)\big\|_{L^\infty(a,b)}\,dt\leq C.
\end{eqnarray*}
\end{lemma}

\begin{proof}
\, By the Sobolev embedding and \eqref{3.2},
 \begin{align*}
  \int_0^T&\big\|\frac{m_x(t,\cdot)}{\rho(t,\cdot)}\big\|_{L^\infty(a,b)}dt
   \leq\,\int_0^T\|m_x(t,\cdot)\|_{L^\infty(a,b)}\|\rho^{-1}(t,\cdot)\|_{L^\infty(a,b)}\,dt\\
  \leq&\, C\int_0^T\Big(\int_a^b|m_{xx}|^2\,dx\Big)^{\frac{1}{2}}\bigg(1+\Big(\int_a^b\frac{|\rho_x|^2}{\rho^3}\,dx\Big)^{\frac{1}{2}}\Big(\int_a^b\phi(\rho)\,dx\Big)^{\frac{1}{2}}\bigg)\,dt\\
  \leq&\, C.
 \end{align*}
 The estimate for $\frac{\rho_x}{\rho}$ is obtained in the same way.

 For $u_x$, note first that $u_x=\frac{m_x}{\rho}-\frac{u\rho_x}{\rho}$ and argue similarly to the above.
 Finally, let $v=\frac{1}{\rho}$. Then, by \eqref{approx}, we calculate
 \begin{align*}
 v_t+\big(u-\varepsilon\frac{A'}{A}\big)v_x-\varepsilon v_{xx}
  = \frac{\rho u_x}{\rho^2}-\frac{2\varepsilon\rho_x^2}{\rho^3}+\frac{A'}{A}\frac{u}{\rho}
  \leq\big(u_x+\frac{A'}{A}u\big)v.
 \end{align*}
By the maximum principle,
$$
\max_{Q_T}v\leq C\max\big\{\|v_0\|_{L^\infty}, \rho_-^{-1}, \rho_+^{-1}\big\}e^{\int_0^T\|(u_x,u)(\tau,\cdot)\|_{L^\infty(a,b)}d\tau}
\leq C \|v_0\|_{L^\infty(a,b)}.
$$
\qed\end{proof}
This completes the proof of Theorem \ref{artsol}.

\section{\, Uniform Estimates}\label{sec:uniform}

As mentioned earlier, the purpose of this section and \S 5--\S 6 is
 to establish the proof of Theorem 2.1 by making the uniform
estimates of the approximate solutions, independent of $\varepsilon$.
Throughout this section, the universal constant $M>0$ is independent of $\varepsilon$,
and we assume that the area function $A(x)\in C^2$ satisfies \eqref{1.4}--\eqref{1.5}.

The first estimate concerns the local higher integrability properties of the density.

\begin{lemma}\label{gamma+1}
Let $K\subset \mathbb{R}$ be a compact set such that $K\subset(a,b)$.
Then, for $T>0$, there exists a constant
$M=M(K,T)$, independent of $\varepsilon$,
such that
\begin{equation}\label{4.1}
 \int_0^T\int_K\big(\rho^{\gamma+1}+\delta\rho^3\big)\,dx\,dt\leq M.
\end{equation}
\end{lemma}

\begin{proof}
 \, The proof of this lemma is standard by now. We simply observe that, on the compact set $K$,
$A$, $\frac{A'}{A}$, and $\frac{A''}{A}$
are bounded uniformly by a constant $M=M(K)$.
Therefore, the argument as in \cite[Lemma 3.3]{ChenPerep2}
(compare also \cite[Lemma 3.3]{ChenPerep1}) yields \eqref{4.1}.
\qed\end{proof}

The other key estimate
concerns a local higher integrability property for the velocity:
$$
\int_0^T\int_K\big(\rho|u|^3+\rho^{\ga+\th}\big)\,dx\,dt\leq M.
$$
The proof of this estimate relies on the observation that a
careful choice of the generating function $\psi=\psi(s)$
yields an entropy flux function of strictly higher growth rate
than its associated weak entropy function.
This observation was first made in \cite{LPT} and
is contained in the next result, which has been
taken and adapted from Chen-Perepelitsa \cite{ChenPerep1,ChenPerep2}.

\begin{lemma}\label{entropy2}
Let $(\eta^\#,q^\#)$
be the entropy pair corresponding to $\psi_\#=\psi_\#(w)=\frac{1}{2}w|w|$.
Then, for $\psi(s)=\psi_\#(s-u_-)$, the associated entropy pair $(\check{\eta},\check{q})$ can be written, for some constant $\alpha$, as
 \begin{align*}
  &\check{\eta}(\rho,m)=\eta^\#(\rho,m-\rho u_-),\\
  &\check{q}(\rho,m)=q^\#(\rho,m-\rho u_-)+u_-\eta^\#(\rho,m-\rho u_-),\\
  &\check{\eta}(\rho,m)=\alpha\rho^{\theta+1}(u-u_-)+r_2(\rho,\rho(u-u_-)),\\
  &|r_2(\rho,\rho(u-u_-))|\leq M\rho|u-u_-|^2.
 \end{align*}
Moreover, for the notations{\rm :}
 \begin{align*}
  &\tilde{\eta}(\rho,m):=\check{\eta}(\rho,m)-\nabla\check{\eta}(\rho_-,m_-)\cdot(\rho-\rho_-,m-m_-),\\
  &\tilde{q}(\rho,m):=\check{q}(\rho,m)-\nabla\check{\eta}(\rho_-,m_-)\cdot(m,\frac{m^2}{\rho}+p),
 \end{align*}
the following inequalities and identities hold{\rm :}
 \begin{align*}
  &|\tilde{\eta}(\rho,m)|\leq M\big(\rho|u-u_-|^2+\rho(\rho^\theta-(\rho_-)^\theta)^2\big),\\[1mm]
  &\tilde{q}(\rho,m)\geq \frac{1}{M}\big(\rho|u-u_-|^3+\rho^{\gamma+\theta}\big)-M\big(\rho+\rho|u-u_-|^2+\rho^\gamma\big),\\[1mm]
  &\tilde{q}(\rho_-,m_-)<0,\\[1mm]
  &\big|-\check{q}+m\check{\eta}_\rho+\frac{m^2}{\rho}\check{\eta}_m\big|\leq M\tilde{q}+M,\\[1mm]
  &\tilde{\eta}_m=\alpha(\rho^\theta-(\rho_-)^\theta)+(u-u_-)r(\rho,u),\hspace{5mm} |r(\rho,u)|\leq M,\\[1mm]
  &|m\tilde{\eta}_m|\leq M\big(\rho(u-u_-)^2+\rho(\rho^\theta-(\rho_-)^\theta)^2+\rho\big),\\[1mm]
  &|(\tilde{\eta}_m)_\rho|\leq M\rho^{\theta-1},\hspace{5mm} |(\tilde{\eta}_m)_u|\leq M,\\[1mm]
  &|\tilde{\eta}_m|\leq M\big(|u-u_-|+|\rho^\theta-(\rho_-)^\theta|\big).
 \end{align*}
\end{lemma}
Before proving the higher integrability of the velocity, we need the following technical lemma.

\begin{lemma}\label{rhocubed}
For any compact set $K\subset (a,b)$,   there exists $\beta>2$ such that
\begin{equation}\label{cubed}
\varepsilon\int_0^T\int_a^x \rho^3(t,y) A(y)\,dy\,dt\leq M|a|^\beta \qquad\, \mbox{for any $x\in K$},
\end{equation}
where $M=M(K)$ is independent of $\varepsilon>0$.
\end{lemma}

\begin{proof}
\, The proof is divided into three cases.

\smallskip
1.  When $\gamma\in(1,2)$,
 \begin{align*}
 &\varepsilon\int_0^T\int_a^x\rho^3 A\,dy\,dt\\
 & \leq M\varepsilon |a|\int_0^T\sup_{(a,x)}\rho^{3-\gamma}\,dt\\
 & \leq M\varepsilon |a|\Big(1+\int_0^T\int_a^x\rho^{3-\frac{3\gamma}{2}}|(\rho^{\frac{\gamma}{2}})_y|\,dy\,dt\Big)\\
 & \leq M\varepsilon |a|\Big(1+\int_0^T\int_a^x\rho^{6-3\gamma}A^{-1}\,dy\,dt
   + \int_0^T\int_a^x\rho^{\gamma-2}|\rho_y|^2A\,dy\,dt\Big)\\
 &  \leq  M\varepsilon |a| \Big(1+\int_0^T\int_a^x\rho^{6-3\gamma}A^{2-\gamma}A^{\gamma-3}\,dy\,dt\Big)\\
 & \leq M\varepsilon |a| \Big(1+\int_0^T\int_a^x\rho^3\Big(\frac{A}{2|a|M}+M|a|^{\beta_1(\gamma)}A^{\frac{\gamma-3}{\gamma-1}}\Big)\,dy\,dt\Big),
 \end{align*}
 where we have used the Young inequality.
 Then
 $$
 \frac{\varepsilon}{2}\int_0^T\int_a^x\rho^3\,dy\,dt\leq M\varepsilon |a|^{\beta(\gamma)},
 $$
 where $\beta(\gamma)=\beta_1(\gamma)+1$.

\smallskip
2. Now we consider the case: $\gamma\in[2,3]$.
Note that  $\int_a^x\rho A\,dy\leq M(E_0)+M |a|\leq M |a|$.
Then
\begin{align*}
  \int_0^T\int_a^x\rho^3A\,dy\,dt &\leq\int_0^T\sup_{(a,x)}\rho^2(t,\cdot)\big(\int_a^x\rho A\,dy\big)\,\,dt\\
  &\leq  M |a|\Big(1+\int_0^T\int_a^x\rho|\rho_y|\,dy\,dt\Big)\\
  &= M|a|\Big(1+\int_0^T\int_a^x\rho^{2-\frac{\gamma}{2}}\rho^{\frac{\gamma-2}{2}}|\rho_y|\,dy\,dt\Big)\\
  &\leq M|a|\Big(1+\frac{1}{\varepsilon}+\int_0^T\int_a^x\rho^{4-\gamma}A^{-1}\,dy\,dt\Big)\\
  &\leq M|a|\Big(1+\varepsilon^{-1}+\int_0^T\int_a^x\big(\rho^\gamma A+A^{-\frac{4}{2\gamma-4}}\big)\,dy\,dt\Big)\\
  &\leq M|a|\big(1+\varepsilon^{-1}+|a|\big),
\end{align*}
where we have written $A^{-1}=A^{\frac{4-\gamma}{\gamma}}A^{-\frac{4}{\gamma}}$ and applied the Young inequality
 with exponents $\frac{\gamma}{4-\gamma}$ and $\frac{\gamma}{2\gamma-4}$.

\smallskip
3. Finally, for $\gamma>3$, the estimate follows directly from the main energy estimate, Proposition \ref{energy}.
\qed\end{proof}

\begin{lemma}\label{hessianbound}
 Let $(\eta^\psi,q^\psi)$ be an entropy pair such that the generating function $\psi$ satisfies
 $$
 \sup_s|\psi''(s)|<\infty.
 $$
 Then, for any  $(\rho,m)\in\mathbb{R}^2$, $\xi\in\mathbb{R}^2$,
 $$
 |\xi^\top\nabla^2\eta^\psi(\rho, m)\xi|\leq M_\psi \xi^\top\nabla^2\eta^*(\rho, m)\xi,
 $$
 where $M_\psi$ depends only on $\psi$ and $\gamma$.
\end{lemma}

The proof is direct; see \cite{ChenPerep1}.
Also recall that the mechanical energy $\eta^*(\rho,m)$ is convex.

The next lemma concerns the higher integrability property for the velocity
and is another crucial estimate
for applying the compensated compactness framework of \cite{ChenPerep1}.

\begin{lemma}\label{unif2}
Let $K\subset(a,b)$ be compact.
Then there exist a constant $M=M(K,T)$ and an exponent
$\beta>2$, both independent of $\varepsilon$, such that
$$
\int_0^T\int_K\big(\rho|u|^3+\rho^{\gamma+\theta}\big)\,dx\,dt
\leq M\Big(1+\frac{|a|^\beta\delta}{\varepsilon}+\varepsilon|a|\Big).
$$
\end{lemma}

\smallskip
\begin{proof}
 \, We prove the property for the nozzles satisfying condition \eqref{1.3a}, since
the other case \eqref{1.3b} can be handled by corresponding similar arguments.

\medskip
We multiply
   the continuity equation in \eqref{approx} by $\tilde{\eta}_\rho A$
 and the momentum equation by
 $\tilde{\eta}_m A$ to obtain that, after a short calculation,
 \begin{align}
  &(\tilde{\eta}A)_t+(\tilde{q}A)_x+A'\big(-\check{q}+m\check{\eta}_\rho+\frac{m^2}{\rho}\check{\eta}_m\big)
  +A'\check{\eta}_m(\rho_-,m_-)p(\rho)\nonumber\\
  &=\, \varepsilon \big(\rho_{xx}+\frac{A'}{A}\rho_x\big)\tilde{\eta}_\rho A
  +\varepsilon \big(m_x+\frac{A'}{A}m\big)_x\tilde{\eta}_mA -(\delta\rho^2)_x\tilde{\eta}_m A. \label{qtildeeqn}
 \end{align}

Integrating both sides of \eqref{qtildeeqn}, we find
\begin{align}
 &\int_0^TA(x)\tilde{q}(x)\,dt\nonumber\\
 &=\int_0^TA(a)\tilde{q}(a)\,dt-\int_a^x(\tilde{\eta}(T,y)-\tilde{\eta}(0,y))A(y)\,dy\nonumber\\
 &\quad -\int_0^T\int_a^xA'(-\check{q}+m\check{\eta}_\rho+\frac{m^2}{\rho}\check{\eta}_m)\,dy\,dt\nonumber\\
 &\quad +\varepsilon\int_0^T\int_a^x \big((\rho_{yy}+\frac{A'}{A}\rho_y)\tilde{\eta}_\rho+(m_y+\frac{A'}{A}m)_y\tilde{\eta}_m\big)A\,dy\,dt\nonumber\\
 &\quad -\delta\int_0^T\int_a^xA(\rho^2)_y\tilde{\eta}_m\,dy\,dt-\int_0^T\int_a^xA'\check{\eta}_m(\rho_-,m_-)p(\rho)\,dy\,dt. \label{mainineq}
 \end{align}

\medskip
Now we employ the crucial observation that the first term, which is poorly controlled in absolute value,
 is actually negative by Lemma \ref{entropy2},
 and so may be neglected.
 Considering the fourth integral on the right hand side, we integrate by parts to see
 \begin{align}
 &\eps\int_0^T\int_a^x\big((\rho_{yy}+\frac{A'}{A}\rho_y)\tilde{\eta}_\rho+(m_{y}+\frac{A'}{A}m)_y\tilde\eta_m\big)A\,dy\,dt\nonumber\\
 & =\eps\int_0^TA(y)\big(\tilde\eta(t,x)\big)_x\,dt-\eps\int_0^T\int_a^xA(y)\big(\rho_y(\tilde\eta_\rho)_y+m_y(\tilde\eta_m)_y\big)\,dy\,dt\nonumber\\
 &\quad +\eps\int_0^T\int_a^x\big(\frac{A'}{A}\big)_ym\tilde\eta_m A(y)\,dy\,dt, \label{4.5a}
 \end{align}

 \smallskip
 \noindent
 as $\tilde{\eta}_\rho=\tilde{\eta}_m=0$ at $x=a$ by construction. Likewise, integrating by parts in the fifth integral,
\begin{align}
\de\int_0^T&\int_a^x(\rho^2)_y\tilde\eta_mA(y)\,dy\,dt\nonumber\\
 =&-\de\int_0^T\int_a^x\big(\rho^2((\tilde\eta_m)_\rho\rho_y+(\tilde\eta_m)_uu_y)A(y)+\rho^2\tilde\eta_mA'(y)\big)\,dy\,dt\nonumber\\
  &+\int_0^T\rho^2\tilde\eta_m A(x)\,dt.\label{4.6a}
 \end{align}

\medskip
 Combining \eqref{mainineq}--\eqref{4.6a}
and applying the bound for $-\check{q}+m\check{\eta}_\rho+\frac{m^2}{\rho}\check{\eta}_m$ from Lemma \ref{entropy2},
 \begin{align}
  \int_0^T\tilde{q}(x)A(x)\,dt\leq&\, I(x)+M\int_0^T\int_a^x|A'(y)|(\tilde{q}(t,y)+1)\,dy\,dt\nonumber\\
  &+\eps\int_0^T\big(A(x)\tilde\eta(t,x)\big)_x\,dt-\de\int_0^T\rho^2\tilde\eta_mA(x)\,dt, \label{4.7a}
\end{align}
where
\begin{align*}
I(x)=&\,M\int_a^xA(x)|\tilde{\eta}(T,y)-\tilde{\eta}(0,y)|\,dy\,dt+\varepsilon\int_0^T\int_a^x\big|A\big(\frac{A'}{A}\big)_ym\tilde{\eta}_m\big|\,dy\,dt\\
&+\varepsilon\int_0^T\int_a^x\big|A\rho_y(\tilde{\eta}_\rho)_y+Am_y(\tilde{\eta}_m)_y\big|\,dy\,dt\\
&+\delta\int_0^T\int_a^x\big|A\rho^2\big((\tilde{\eta}_m)_\rho\rho_y+(\tilde{\eta}_m)_uu_y\big)\big|\,dy\,dt\\
&+\delta\int_0^T\int_a^x\big|A'\rho^2\tilde{\eta}_m\big|\,dy\,dt+\int_0^T\int_a^x\big|A'\check{\eta}_m(\rho_-,m_-)p(\rho)\big|\,dy\,dt\\
&+\varepsilon\int_0^T|A'||\tilde{\eta}(t,x)|\,dt\\
=&\,I_1+\cdots+I_7.
\end{align*}
Note that $I$ is increasing in $x$ and $|\frac{A'}{A}|\leq M$. By the Gronwall inequality,
\begin{align}
 \int_0^TA\tilde{q}(x)\,dx\leq&\, I(x)e^{\int_a^x|A'|\,dy}+\varepsilon\int_0^T\big(A\tilde{\eta}(t,x)\big)_x\,dt-\int_0^T\delta\rho^2\tilde{\eta}_mA\,dt\nonumber\\
 &+\int_a^x\Big(\varepsilon\big(\int_0^TA\tilde{\eta}(t,y)\,dt\big)_y-\int_0^T\delta\rho^2\tilde{\eta}_mA\,dt\Big)|A'|e^{\int_y^x|A'|\,ds}\,dy.\nonumber\\
 &\label{eq:frog}
\end{align}
We take a smooth cut-off function $\om(x)$ with compact support in $K$ such that $\om(x)\leq 1$.
Observe first that, by the bound $|\tilde\eta(\rho,u)|\leq M(\rho|u-u_-|^2+\rho(\rho^\th-(\rho_-)^\th)^2)$, we obtain, from Proposition \ref{energy},
\beq\label{ineq:etatilde}
\int_a^x\tilde\eta(t,y)A(y)\,dy\,dt\leq M(E_0+1)  \qquad\,\,\,\mbox{for any $x\in \supp(\omega)$}.
\eeq
Thus, we may integrate \eqref{eq:frog} against $\om$ and estimate the terms on the right as follows:
\begin{align}
&\Big|\int_K\omega\int_a^x\varepsilon\big(\int_0^TA\tilde{\eta}(t,y)\,dt\big)_y|A'|e^{\int_y^x|A'|\,ds}\,dy\,dx\Big|\nonumber\\
&\quad\leq \eps M(\|A'\|_{L^1},\|A''\|_{L^\infty})\int_K\omega\int_0^T\int_a^xA(y)\tilde{\eta}(t,y)\,dy\,dt\,dx+M\leq \eps M|a|,\nonumber\\
& \label{ineq:integral1}\\
&\Big|\int_a^x\hspace{-0.7mm}\int_0^T\delta\rho^2\tilde{\eta}_mA\,dt|A'|e^{\int_x^y|A'|ds}dy\Big|\leq M\int_0^T\hspace{-0.7mm}\int_a^x\delta\rho^3A\,\hspace{-0.25mm}dy\,\hspace{-0.25mm}dt\leq \frac{M|a|^\beta\delta}{\varepsilon},\label{ineq:integral2}\\
&\Big|\varepsilon\int_K\omega\int_0^T\big(A\tilde{\eta}(t,x)\big)_x\,dt\,dx\Big|=\varepsilon\int_0^T\int_K\big|\omega_xA\tilde{\eta}(t,x)\big|\,dx\,dt
\leq M,\label{ineq:integral3}
\end{align}
where we have used Lemma \ref{rhocubed}.

Therefore, the lower bound $\check{q}\geq \frac{1}{M}(\rho|u-u_-|^3+\rho^{\gamma+\theta})-M(\rho+\rho|u-u_-|^2+\rho^\gamma)$ yields
that, after multiplying by $\om$ and integrating in $x$, we have
\begin{align*}
\int_K&\int_0^T\omega A\big(\rho|u-u_-|^3+\rho^{\gamma+\theta}\big)\,dt\,dx\\
&\leq\int_K\omega(x)I(x)e^{\int_a^x|A'|\,ds}\,dx
+M\int_K\omega\int_0^T\big(\rho|u-u_-|^2+\rho^\gamma+\rho\big)\,dt\,dx\\
&\quad +M\big(1+\frac{|a|^\be\de}{\eps}+\eps |a|\big).
\end{align*}
 For the term involving $I(x)$,
we have seen that $\int_K|I_1|\,dx\leq M$ by \eqref{ineq:etatilde}.
Now we can use the inequality for $|m\tilde{\eta}_m|$ from Lemma \ref{entropy2} to obtain
\begin{align*}
&\int_K\omega |I_2|\,dx \\
&\leq M \int_K\omega\varepsilon\int_0^T\int_a^x\big|A\big(\frac{A'}{A}\big)_ym\tilde{\eta}_m\big|\,dy\,dt\,dx\\
&\leq M\int_K\omega\varepsilon\int_0^T\int_a^x\big|A\big(\frac{A'}{A}\big)_y\big(\rho+\rho(u-u_-)^2+\rho(\rho^\theta-(\rho_-)^\theta)^2\big)\big|\,dy\,dt\,dx\\
&\leq M\int_K\omega\varepsilon\int_0^T\int_a^x\big|A\big(\frac{A'}{A}\big)_y(1+\overline{\eta_\delta^*})\big|\,dy\,dt\,dx\\
&\leq \varepsilon M\big\|A\big(\frac{A'}{A}\big)'\big\|_{L^\infty(a,x)}|a|\\
& \leq \varepsilon M |a|,
\end{align*}
where we have used that $x\in K$ and $M$ depends on $K$.

Moreover, we have
\begin{align*}
|I_3|=&\,\varepsilon\int_0^T\int_a^x|A\rho_y(\tilde{\eta}_\rho)_y+Am_y(\tilde{\eta}_m)_y|\,dy\,dt\\
=&\,\varepsilon\int_0^T\int_a^xA|(\rho_y,m_y)\nabla^2\tilde{\eta}(\rho_y,m_y)^\top|\,dy\,dt\\
\leq&\, M_\psi\varepsilon\int_0^T\int_a^xA(\rho_y,m_y)\nabla^2\overline{\eta^*}(\rho_y,m_y)^\top\,dy\,dt\\
\leq&\, ME_0,
\end{align*}
by Proposition \ref{energy} and Lemma \ref{hessianbound}.

Next, also by Lemmas \ref{entropy2}--\ref{rhocubed}, we obtain
\begin{align*}
|I_4|=&\,\delta\int_0^T\int_a^x\big|A\rho^2\big((\tilde{\eta}_m)_\rho\rho_y+(\tilde{\eta}_m)_uu_y\big)\big|\,dy\,dt\\
\leq&\,\delta\int_0^T\int_a^xA\rho^2\big(M\rho^{\theta-1}|\rho_y|+|u_y|\big)\,dy\,dt\\
\leq&\,\delta M\big(\int_0^T\int_a^x A\rho^3\,dy\,dt\big)^{\frac{1}{2}}\big(\int_0^T\int_a^xA\rho^{\gamma-2}|\rho_y|^2\,dy\,dt\big)^{\frac{1}{2}}\\
&\,+\delta M\big(\int_0^T\int_a^xA\rho^3\,dy\,dt\big)^{\frac{1}{2}}\big(\int_0^T\int_a^xA\rho|u_y|^2\,dy\,dt\big)^{\frac{1}{2}}\\
\leq&\,\frac{\delta}{\sqrt{\varepsilon}}M\big(\int_0^T\int_a^xA\rho^3\,dy\,dt\big)^{\frac{1}{2}}\\
\leq&\,M\frac{|a|^\beta\delta}{\varepsilon}.
\end{align*}
Again, from Lemma \ref{entropy2},
\begin{align*}
|I_5|\leq&\,M\delta\int_0^T\int_a^x|A'|\rho^2\big(\alpha|\rho^\theta-(\rho_-)^\theta|+M|u-u_-|\big)\,dy\,dt\\
\leq&\,\delta M\int_0^T\int_a^xA(y)\big(\rho^3+\rho(\rho^\theta-(\rho_-)^\theta)^2+\rho|u-u_-|^2\big)\,dy\,dt\\
\leq&\,\delta\big(M+M\int_0^T\int_a^xA\rho^3\,dy\,dt\big)\\
\leq&\, M+M\frac{|a|^\beta\delta}{\varepsilon}.
\end{align*}
Since $|\frac{A'}{A}|$ is uniformly bounded,
\begin{align*}
|I_6|=&\,\int_0^T\int_a^x|A'\check{\eta}_m(\rho_-,m_-)p(\rho)|\,dy\leq M\int_0^T\int_a^x(A\overline{h_\de}(\rho,\bar{\rho})+A')\,dy\\
\leq&\, MT(E_0+\|A'\|_{L^1(-\infty,\sup(K))}),
\end{align*}
and
\begin{align*}
\int_K|I_7|\,dx=\int_K\omega\int_0^T|A'||\tilde{\eta}|\,dt\,dx\leq M\int_0^T\int_{\supp\,\omega}A(y)|\tilde{\eta}(t,x)|\,dx\,dt\leq M.
\end{align*}
These estimates for $I_j$, $j=1,...7$, show that
$$\int_K\omega(x)I(x)e^{\int_a^x|A'|\,ds}\,dx\leq M,$$
and hence we conclude the proof.
\qed\end{proof}

Once we have shown the above estimates to be uniform, we can apply the compensated compactness techniques from
\cite{ChenPerep1}, as in \cite{ChenPerep2}.
To make this rigorous, we require our approximate solutions in the construction to satisfy the following:
\begin{enumerate}
  \item[(i)] $(\rho_0^\varepsilon,m_0^\varepsilon)\rightarrow(\rho_0,m_0)$ a.e. $x\in\mathbb{R}$ as $\varepsilon\rightarrow 0$,
   where we take $(\rho_0^\varepsilon,m_0^\varepsilon)$ to be the zero extension of $(\rho_0^\varepsilon,m_0^\varepsilon)$
   outside $(a,b)$;\

\medskip
\item[(ii)] $\int_a^b\overline{\eta_\delta^*}(\rho_0^\varepsilon(x), m_0^\varepsilon(x))A(x)\,dx
  \to \int_\mathbb{R}\overline{\eta^*}(\rho_0(x), m_0(x))A(x)\,dx$
       as $\varepsilon\rightarrow 0$;

\medskip
\item[(iii)] $\frac{\delta|a|^\beta}{\varepsilon}\leq M$;

\smallskip
\item[(iv)] $\varepsilon|b-a|\leq M$,
\end{enumerate}
where $M<\infty$ is independent of $\varepsilon\in(0,1]$.

With these properties of the approximate solutions, the bounds in the above
lemmas become uniform in $\varepsilon$.

\begin{proposition}\label{compactness}
 Let $(\eta,q)$ be the entropy pair generated by $\psi\in C_c^\infty(\mathbb{R})$. Then the entropy dissipation
 measures{\rm :}
 $$
 \eta(\rho^\varepsilon,m^\varepsilon)_t+q(\rho^\varepsilon,m^\varepsilon)_x,
 $$
 lie in a compact subset of $H^{-1}_{loc}$.
\end{proposition}

\begin{proof}
\, We divide the proof into six steps.

1. We first recall the following fact from Lemma 2.1 in \cite{ChenPerep1}:
 For a $C^2$--function $\psi:\mathbb{R}\rightarrow\mathbb{R}$ of compact support,
 the associated entropy pair $(\eta, q)$ satisfies
 \begin{align*}
  &|\eta(\rho,m)|\leq M_\psi \rho,\\[1mm]
  &|q(\rho,m)|\leq M_\psi \rho \qquad\qquad\qquad\quad\mbox{for $\gamma\in (1, 3]$},\\[1mm]
  &|q(\rho,m)|\leq M_\psi \rho\max\{1, \rho^\theta\} \qquad\mbox{for $\gamma>3$},\\[1mm]
  &|\eta_m(\rho,m)|+|\rho\eta_{mm}(\rho,m)|\leq M_\psi,\\[1mm]
  &|\eta_\rho(\rho,m)+u\eta_m(\rho,m)|\leq M_\psi(1+\rho^\theta),\\[1mm]
  &|\eta_{mu}(\rho,\rho u)|+|\rho^{1-\theta}\eta_{m\rho}(\rho,\rho u)|\leq M_\psi,
 \end{align*}
 where, in the last inequality, we regard $\eta_m$ as a function of $(\rho,u)$.

\smallskip
2.  Write $\eta^\varepsilon=\eta(\rho^\varepsilon,m^\varepsilon)$ and $q^\varepsilon=q(\rho^\varepsilon,m^\varepsilon)$
with $m^\varepsilon=\rho^\varepsilon u^\varepsilon$. Then
 \begin{align}
  \eta^\varepsilon_t+q^\varepsilon_x
  =&-\frac{A'}{A}\rho^\varepsilon u^\varepsilon\big(\eta^\varepsilon_\rho+u^\varepsilon\eta^\varepsilon_m\big)
  +\varepsilon\big(\frac{A'}{A}\rho^\varepsilon_x\eta^\varepsilon_\rho+\big(\frac{A'}{A}m^\varepsilon\big)_x\eta^\varepsilon_m\big) \nonumber\\
  &-\varepsilon\big(\rho^\varepsilon_x(\eta^\varepsilon_\rho)_x+m^\varepsilon_x(\eta^\varepsilon_m)_x\big)
   +\varepsilon\eta^\varepsilon_{xx}-\big(\delta(\rho^\varepsilon)^2\big)_x\eta^\varepsilon_m \nonumber\\
  =&:I^\varepsilon_1+\cdots+I^\varepsilon_5. \label{div(eta,q)}
 \end{align}
 We want to prove that this is a sum of terms bounded uniformly in the space $L^1(0,T;L^1_{loc}(\mathbb{R}))$ and terms that are compact
 in $W^{-1,q}_{loc}(\mathbb{R}^2_+)$ for some $q>1$.
 Then the compact Sobolev embedding of $L^1$ into $W^{-1,q}$
 (locally in $\mathbb{R}_+^2$) gives the compactness of the entropy dissipation
 measures in some $W^{-1,q_1}_{loc}(\mathbb{R}^2_+)$.

 \smallskip
 3. To this end, first observe that
 \begin{align*}
  |I^\varepsilon_1(t,x)|\leq&\, M\big|\frac{A'}{A}\big|\rho^\varepsilon\big(1+(\rho^\varepsilon)^\theta\big)|u^\varepsilon|
  \leq \, M \big|\frac{A'}{A}\big| \big(\rho^\varepsilon|u^\varepsilon|^2+\rho^\varepsilon+(\rho^\varepsilon)^\gamma\big)
 \end{align*}
 is uniformly bounded in $L^1(0,T;L^1_{loc}(\mathbb{R}))$.

 Now we see
 \begin{align*}
  I^\varepsilon_2=&\,\varepsilon\big(\frac{A'}{A}\eta^\varepsilon_x+\big(\frac{A'}{A}\big)_xm^\varepsilon\eta^\varepsilon_m\big)
  =\varepsilon\big(\frac{A'}{A}\big)_x(m\eta^\varepsilon_m-\eta^\varepsilon)+\varepsilon\big(\frac{A'}{A}\eta^\varepsilon\big)_x\\[1mm]
  =&:I^\varepsilon_{2a}+I^\varepsilon_{2b}.
 \end{align*}
 One can easily check that $|\eta^\varepsilon-m^\varepsilon\eta^\varepsilon_m|\leq M\big(\rho^\varepsilon+\rho^\varepsilon|u^\varepsilon|^2\big)$.
 Thus, $I^\varepsilon_{2a}\rightarrow 0$ in $L^1_{loc}(\mathbb{R}^2_+)$ as $\varepsilon\rightarrow 0$.
 Next, for $\omega\in C_c^\infty(\mathbb{R}^2_+)$,
 \begin{align*}
  \varepsilon\big|\int_{\supp\,\omega}I^\varepsilon_{2b}\omega(t,x)\,dx\,dt\big|
  &=\big|\varepsilon\int_{\supp\,\omega}\frac{A'}{A}\eta^\varepsilon\omega_x\,dx\,dt\big|\\
  &\leq \varepsilon M(\supp\,\omega)\|\rho^\varepsilon\|_{L^{\gamma+1}(\supp\,\omega)}\|\omega\|_{H^1(\mathbb{R}^2_+)}.
 \end{align*}
 Hence, $I^\varepsilon_{2b}\rightarrow 0$ in $H^{-1}_{loc}(\mathbb{R}^2_+)$ as $\varepsilon\rightarrow 0$.

\smallskip
 Moreover,
 \begin{align*}
  |I^\varepsilon_3|=&\,\varepsilon|\langle\nabla^2\eta(\rho^\varepsilon,m^\varepsilon)(\rho^\varepsilon_x,m^\varepsilon_x),(\rho^\varepsilon_x,m^\varepsilon_x)\rangle|\\
  \leq&\, M_\psi\varepsilon\langle\nabla^2\overline{\eta^*}(\rho^\varepsilon,m^\varepsilon)(\rho^\varepsilon_x,m^\varepsilon_x),(\rho^\varepsilon_x,m^\varepsilon_x)\rangle,
 \end{align*}
 which is uniformly bounded in $L^1(0,T;L^1_{loc}(\mathbb{R}))$ by the main energy estimate, Proposition \ref{energy}.

\bigskip
4. The next step is to show that $I_4^\varepsilon\rightarrow 0$ in $W^{-1,q}_{loc}$ for some $q>1$.
 \begin{claim}
  \, Let $K\subset\mathbb{R}$ be compact, and let $0<\Delta<1$. Then
 $$
 \int_0^T\int_K\varepsilon^{\frac{3}{2}}|\rho^\varepsilon_x|^2\,dx\,dt\leq M(T,\supp\,\omega)\big(\De+\eps+\eps\De^{4-\ga}+\frac{\eps^2}{\De}\big).
 $$
  Hence
  $$
  \int_0^T\int_K\varepsilon^{\frac{3}{2}}|\rho^\varepsilon_x|^2\,dx\,dt\rightarrow 0,
  $$
  and
  $$
  \varepsilon\eta^\varepsilon_x\rightarrow 0 \qquad\, \text{ in $L^q(0,T;L^q_{loc}(\mathbb{R}))$ for some $q\in(1,2)$.}
  $$
 \end{claim}
 We now drop the superscript $\varepsilon$ and prove the claim.
 Define
 \begin{equation*}
  \phi(\rho)=\begin{cases}
   \frac{\rho^2}{2}, &\rho<\Delta,\\[1mm]
   \frac{\Delta^2}{2}+\Delta(\rho-\Delta), &\rho\geq\Delta,
  \end{cases}
 \end{equation*}
 so that $\phi''(\rho)=\chi_{\{\rho<\Delta\}}$, and
 \begin{equation*}
  \rho\phi'(\rho)-\phi(\rho)=\begin{cases}
   \frac{\rho^2}{2}, &\rho<\Delta,\\[1mm]
   \frac{\Delta^2}{2}, &\rho\geq\Delta.
  \end{cases}
 \end{equation*}
 Let $\omega=\omega(x)\geq 0$ and $\omega\in C_c^\infty(\mathbb{R})$. From the approximate continuity equation \eqref{approx},
 \begin{align*}
 &(\phi\omega^2)_t+(\phi u\omega^2)_x-2\phi u\omega_x\omega\\
 &=\phi'(\rho)\rho_t\omega^2+\phi'(\rho)\rho_x u\omega^2+\phi(\rho)u_x\omega^2\\
  &=-\phi'(\rho)\rho_x u\omega^2-\phi'(\rho)\rho u_x\omega^2+\phi'(\rho)\rho_x u\omega^2 +\phi(\rho)u_x\omega^2\\
  &\quad -\frac{A'}{A}\phi'(\rho)\rho u\omega^2 +\varepsilon\phi'(\rho)\big(\rho_{xx}+\frac{A'}{A}\rho_x\big)\omega^2\\
  &=-\frac{A'}{A}\rho u\omega^2\min\{\rho,\Delta\}-\frac{1}{2}u_x\omega^2\big(\rho^2\chi_{\{\rho<\Delta\}}+\Delta^2\chi_{\{\rho>\Delta\}}\big)+\varepsilon(\phi'\omega^2\rho_x)_x\\
  &\quad -\varepsilon\omega^2|\rho_x|^2\chi_{\{\rho<\Delta\}}-2\varepsilon\min\{\rho,\Delta\}\omega_x\rho_x\omega+\varepsilon\frac{A'}{A}\min\{\rho,\Delta\}\rho_x\omega^2.
 \end{align*}

 Thus, by integrating, we have
 \beqas
  \int_0^T\int&\varepsilon\omega^2|\rho_x|^2\chi_{\{\rho<\Delta\}}\,dx\,dt\\
  =&-\int\phi\omega^2\vert_0^T\,dx+2\int_0^T\int\phi u\omega_x\omega \,dx\,dt\\
  &-\frac{1}{2}\int_0^T\int(\rho^2\chi_{\{\rho<\Delta\}}+\Delta^2\chi_{\{\rho>\Delta\}})\omega^2u_x\,dx\,dt\\
  &-\int_0^T\int\frac{A'}{A}\rho u\omega^2\min\{\rho,\Delta\}\,dx\,dt-2\int_0^T\int\varepsilon\min\{\rho,\Delta\}\omega_x\omega\rho_x\,dx\,dt\\
  &+\int_0^T\int\varepsilon\frac{A'}{A}\min\{\rho,\Delta\}\rho_x\omega^2\,dx\,dt\\
  =&:J_1+\cdots +J_6.
 \eeqas
 The first four terms, $J_1,...,J_4$, are bounded by using the uniform energy estimates, giving a bound of $M\frac{\De}{\sqrt{\eps}}$.
 We therefore focus on the last two terms, $J_5$ and $J_6$.
 \beqas
  |J_5|\leq&\,2\sqrt{\varepsilon}\int_0^T\int\sqrt{\varepsilon}\rho\chi_{\{\rho<\Delta\}}|\omega_x|\omega|\rho_x|\,dx\,dt\\
    &+2\sqrt{\varepsilon}\int_0^T\int\sqrt{\varepsilon}\Delta\chi_{\{\rho>\Delta\}}|\omega_x|\omega|\rho_x|\,dx\,dt\\
  \leq&\,\frac{\varepsilon}{4}\int_0^T\int|\rho_x|^2\chi_{\{\rho<\Delta|\}}\omega^2\,dx\,dt+M\varepsilon\int_0^T\int\Delta^2|\omega_x|^2\,dx\,dt\\
  &\,+M\sqrt{\varepsilon}\int_0^T\int\sqrt{\varepsilon}|\rho_x|\rho^{\frac{\gamma-2}{2}}\chi_{\{\rho>\Delta\}}\rho^{\frac{2-\gamma}{2}}\Delta|\omega_x|\omega \,dx\,dt\\
  \leq&\,\frac{\varepsilon}{4}\int_0^T\int|\rho_x|^2\chi_{\{\rho<\Delta|\}}\omega^2\,dx\,dt+\varepsilon\Delta^2M\\
  &\,+M\sqrt{\varepsilon}\int_0^T\int\varepsilon\rho^{\gamma-2}|\rho_x|^2\omega^2\,dx\,dt\\
  &\,+M\sqrt{\varepsilon}\int_0^T\int\chi_{\{\rho>\Delta\}}\rho^{2-\gamma}\Delta^2\omega_x^2\,dx\,dt\\
  \leq&\,\frac{\varepsilon}{4}\int_0^T\int|\rho_x|^2\chi_{\{\rho<\Delta\}}\omega^2\,dx\,dt+\varepsilon\Delta^2M+\sqrt{\varepsilon}M\\
  &\,+M\sqrt{\varepsilon}\int_0^T\int\Delta\chi_{\{\Delta<\rho\}}\rho^{3-\gamma}\omega_x^2\,dx\,dt\\
  \leq&\,\frac{\varepsilon}{4}\int_0^T\int|\rho_x|^2\chi_{\{\rho<\Delta\}}\omega^2\,dx\,dt+\sqrt{\varepsilon}M+\sqrt{\eps}\De^{4-\ga}M,
  \eeqas
  where we have used that $\rho^{3-\ga}\leq \max\{\rho^{\ga+1},1\}$ and  Lemma \ref{gamma+1} when $\ga\in(1,3]$, and
  $\rho^{3-\ga}\leq\De^{3-\ga}$ on $\{\De<\rho\}$ when $\ga>3$.

  Next, using an argument similar to that for $J_5$ and the bounds from Proposition \ref{energy}, we have
  \beqas
  |J_6|\leq&\,\varepsilon\int_0^T\int\big|\frac{A'}{A}\big|\omega^2\rho\chi_{\{\rho<\Delta\}}|\rho_x|dx\,dt+\varepsilon\int_0^T\int\big|\frac{A'}{A}\big|\omega^2\Delta\chi_{\{\rho>\Delta\}}|\rho_x|\,dx\,dt\\
  \leq&\,\frac{\varepsilon}{4}\int_0^T\int|\rho_x|^2\omega^2\chi_{\{\rho<\Delta\}}\,dx\,dt+M\varepsilon\Delta^2\big\|\frac{A'}{A}\big\|_{L^\infty(\supp\,\omega)}|\supp\,\omega\times[0,T]|\\
  &\,+M\sqrt{\varepsilon}\int_0^T\int\sqrt{\varepsilon}\big|\frac{A'}{A}\big||\rho_x|\rho^{\frac{\gamma-2}{2}}\chi_{\{\Delta<\rho\}}\rho^{\frac{2-\gamma}{2}}\Delta\omega^2\,dx\,dt\\
  \leq&\,\frac{\varepsilon}{4}\int_0^T\int|\rho_x|^2\omega^2\chi_{\{\rho<\Delta\}}\,dx\,dt+\sqrt{\varepsilon}\Delta M+\sqrt{\eps}\De^{4-\ga}M.
 \eeqas
 Thus,
 $$
 \int_0^T\int\varepsilon\omega^2|\rho_x|^2\chi_{\{\rho<\Delta\}}\,dx\,dt\leq M\Big(\sqrt{\eps}(1+\De^{4-\ga})+\frac{\De}{\sqrt{\varepsilon}}\Big).
 $$
Moreover, we have
 \begin{align*}
  &\int_0^T\int \varepsilon\omega^2|\rho_x|^2\chi_{\{\rho>\Delta\}}\,dx\,dt\\
  &\leq M\sqrt{\varepsilon}\int_0^T\int\varepsilon\rho^{\gamma-2}|\rho_x|^2\omega^2\chi_{\{\Delta<\rho\}}\,dx\,dt+M\sqrt{\varepsilon}\int_0^T\int\varepsilon\chi_{\{\Delta<\rho\}}\rho^{2-\gamma}\omega^2dx\,dt\\
  &\leq \sqrt{\varepsilon} M+\sqrt{\varepsilon}M\int_0^T\int\varepsilon\chi_{\{\Delta<\rho\}}\omega^2(1+\rho^{\gamma+1})\,dx\,dt\leq\sqrt{\varepsilon}M+\frac{\varepsilon^{\frac{3}{2}}M}{\Delta}.
 \end{align*}
 Therefore, we conclude
 $$
 \int_0^T\int\varepsilon^{\frac{3}{2}}|\rho_x|^2\omega^2\,dx\,dt\leq M\big(\De+\eps+\eps\De^{4-\ga}+\frac{\eps^2}{\De}\big),
 $$
 so that
 $$\int_0^T\int\varepsilon^{\frac{3}{2}}|\rho_x|^2\,dx\,dt\rightarrow 0 \qquad \text{ as $\varepsilon\rightarrow 0$.}$$

 For the second part of the claim, note that
 $$
 |\eta_x|\leq M\big(|\rho_x||\eta_\rho+u\eta_m|+\rho|u_x\eta_m|\big)
    \leq M\big(|\rho_x|(1+\rho^\theta)+\rho|u_x|\big).
 $$
 Let $q\in(1,2)$ be such that $\frac{q}{2-q}=\gamma+1$.
 Then
 \begin{align*}
  \int_0^T\int_K\varepsilon^q|\eta_x|^q\,dx\,dt
  \leq&\, M\int_0^T\int_K\varepsilon^q|\rho_x|^q\,dx\,dt\\
  &+\int_0^T\int_K\varepsilon^q\big||\rho_x|\rho^\theta+\rho|u_x|\big|^q\,dx\,dt\\
  \leq&\,\Delta+\frac{M}{\Delta}\int_0^T\int_K\varepsilon^2|\rho_x|^2\,dx\,dt\\
  &+M\int_0^T\int_K\varepsilon^q\rho^{\frac{q}{2}}\big(\big|\rho^{\frac{\gamma-2}{2}}\rho_x\big|^q+|\sqrt{\rho}u_x|^q\big)\,dx\,dt\\
  \leq&\,\Delta+\frac{M\sqrt{\varepsilon}}{\Delta}\int_0^T\int_K\varepsilon^{\frac{3}{2}}|\rho_x|^2\,dx\,dt\\
  &+\varepsilon^{q-1}M\int_0^T\int_K\varepsilon\big(\rho^{\gamma-2}|\rho_x|^2+\rho|u_x|^2\big)\,dx\,dt\\
  &+\varepsilon^{q-1}M\int_0^T\int_K\varepsilon\rho^{\frac{q}{2-q}}\,dx\,dt\\
  \leq&\,\Delta+\frac{M\sqrt{\varepsilon}}{\Delta}\int_0^T\int_K\varepsilon^{\frac{3}{2}}|\rho_x|^2\,dx\,dt+\varepsilon^{q-1}M.
 \end{align*}
 Thus, $I^\varepsilon_4=\varepsilon\eta^\varepsilon_{xx}\rightarrow 0$ in $W^{-1,q}_{loc}$.

\smallskip
 5. Consider finally $I^\varepsilon_5=\big(\delta(\rho^\varepsilon)^2\big)_x\eta^\varepsilon_m.$ In the case $\ga\leq 3$,
 \begin{align*}
  \int_0^T\int_K|I^\varepsilon_5|\,dx\,dt\leq&\, M_\psi\int_0^T\int_K\delta\rho|\rho_x|\,dx\,dt\\
  \leq&\, M \int_0^T\int_K\big(\varepsilon\rho^{\gamma-2}|\rho_x|^2+\frac{\delta^2}{\varepsilon}\rho^{4-\gamma}\big)\,dx\,dt\\
  \leq&\, M\Big(\int_0^T\int_K\frac{\delta^2}{\varepsilon}\rho^3\,dx\,dt +\frac{\delta^2}{\varepsilon}+1\Big)\\
  \leq&\, M\Big(\frac{\delta^2}{\varepsilon}+\frac{\delta}{\varepsilon}+1\Big),
 \end{align*}
 which is uniformly bounded in $\varepsilon$.

 If $\ga>3$, consider the above on the domains: $\{\rho<1\}$ and $\{\rho>1\}$.
 On the latter domain, we argue as above to obtain
 $$
 \int_0^T\int_{K\cap\{\rho>1\}}|I^\varepsilon_5|\leq M,
 $$
so that
 \beqas
 \int_0^T\int_K|I^\varepsilon_5|\,dx\,dt\leq&\, M+M_\psi\int_0^T\int_{K\cap\{\rho<1\}}\delta\rho|\rho_x|\,dx\,dt\nonumber  \\
 \leq&\,M+M(T|K|)^\half\Big(\int_0^T\int_{K\cap\{\rho<1\}}\de^2\rho^2\rho_x^2\,dx\,dt\Big)^\half \nonumber\\
 \leq&\,M+M\Big(\int_0^T\int_K\eps\de\rho_x^2\,dx\,dt\Big)^\half.  \nonumber
 \eeqas
Thus, $I^\varepsilon_5$ is uniformly bounded in $L^1(0,T;L^1_{loc}(\mathbb{R}))$.

\smallskip
6. In Steps 1--5 above, we have shown that $\eta(\rho^\varepsilon, m^\varepsilon)_t+q(\rho^\varepsilon, m^\varepsilon)_x$ can
be written as
$$
\eta(\rho^\varepsilon, m^\varepsilon)_t+q(\rho^\varepsilon, m^\varepsilon)_x:=f^\varepsilon+g^\varepsilon,
$$
where $f^\varepsilon$ is uniformly bounded in $L^1(0,T;L^1_{loc}(\mathbb{R}))$, and $g^\varepsilon\rightarrow 0$
in $W^{-1,q}_{loc}(\mathbb{R}^2_+)$ for some $q\in(1,2)$.
Thus, there exists $q_1\in(1,2)$ such that
$$
\eta(\rho^\varepsilon, m^\varepsilon)_t+q(\rho^\varepsilon, m^\varepsilon)_x \qquad \text{ is pre-compact in $W^{-1,q_1}_{loc}(\mathbb{R}^2_+)$.}
$$
We also know from Lemmas \ref{gamma+1} and $4.5$ that $\eta(\rho^\varepsilon, m^\varepsilon)$ and $q(\rho^\varepsilon, m^\varepsilon)$
are uniformly bounded in
$L^{q_2}_{loc}(\mathbb{R}^2_+)$,
where $q_2=\gamma+1>2$ when $\gamma\in (1,3]$ and $q_2=\frac{\gamma+\theta}{1+\theta}>2$ when $\gamma>3$.
Then, by the compensated compactness interpolation theorem,
\beqs
\eta(\rho^\varepsilon, m^\varepsilon)_t+q(\rho^\varepsilon, m^\varepsilon)_x \qquad \text{ is pre-compact in $W^{-1,2}_{loc}(\mathbb{R}^2_+)$.}\tag*{\qed}
\eeqs
\end{proof}

\section{\, Convergence to Entropy Solutions}\label{sec:convergence}
Proposition \ref{compactness}, combined with the uniform estimates above,
implies that the sequence of approximate solutions
satisfies the compensated compactness framework in Chen-Perepelitsa \cite{ChenPerep1}.
With this framework, specifically the results
of \S 5 and \S 7 in \cite{ChenPerep1}, we conclude that
\begin{equation*}
 (\rho^\varepsilon,m^\varepsilon)\rightarrow(\rho,m) \qquad \text{{\it a.e.} $(t,x)\in\mathbb{R}^2_+$,}
\end{equation*}
and also in $L^p_{loc}(\mathbb{R}^2_+)\times L^q_{loc}(\mathbb{R}^2_+)$, for $p\in[1,\gamma+1)$
and $q\in[1,\frac{3(\gamma+1)}{\gamma+3})$.
This restriction may be seen as $|m|^q=\rho^{\frac{q}{3}}|u|^q\rho^{\frac{2q}{3}}\leq \rho|u|^3+\rho^{\gamma+1}$,
precisely when $q=\frac{3(\gamma+1)}{\gamma+3}$.

The uniform estimates also give us the convergence of the mechanical energy as $\varepsilon\rightarrow0$:
$$
\eta^*(\rho^\varepsilon,m^\varepsilon)\rightarrow \eta^*(\rho,m) \qquad \text{ in $L^1_{loc}(\mathbb{R}^2_+)$.}
$$
Moreover, from \eqref{eta*}, we have
\begin{equation*}
 \int_{t_1}^{t_2}\int_\mathbb{R}\eta^*(\rho,m)(t,x)A(x)\,dx\,dt\leq M(t_2-t_1)\int_\mathbb{R}\eta^*(\rho_0,m_0)(x)A(x)\,dx+M,
\end{equation*}
so that, for a.e. $t\geq 0$,
\begin{equation*}
 \int_\mathbb{R}\eta^*(\rho,m)(t,x)A(x)\,dx\leq M\int_\mathbb{R}\eta^*(\rho_0,m_0)(x)A(x)\,dx+M.
 \end{equation*}
This implies that no concentration of the density $\rho$ is formed at any point in finite time.

We also see that, for any convex $\psi(s)$ with sub-quadratic growth, the uniform energy estimates of
Lemmas \ref{gamma+1}--\ref{entropy2} and the main energy estimate, Proposition \ref{energy},
give that the sequences:
\begin{equation*}
 \eta^\psi(\rho^\varepsilon, m^\varepsilon),\hspace{2mm} q^\psi(\rho^\varepsilon, m^\varepsilon),\hspace{2mm} m^\varepsilon\eta^\psi_\rho(\rho^\varepsilon, m^\varepsilon)
\end{equation*}
are equi-integrable. Thus, given such a function $\psi$,
we multiply \eqref{div(eta,q)} by $A(x)$ and integrate against a smooth,
compactly supported test function in $\mathbb{R}^2_+$,
and then pass to the limit to obtain the entropy inequality
\eqref{entropyinequality}.

It remains only to check the limit in the equations in the sense of distributions.
This is straightforward
by multiplying the equations of $\eqref{approx}$ by a test function and integrating by parts before
passing to the limit and using the uniform bounds above.
Note that, in the momentum equation, the term
$\delta \rho^2$ vanishes in the limit as $\delta=\delta(\varepsilon)\rightarrow 0$ as $\varepsilon\rightarrow 0$.
This completes the proof of the main theorem, Theorem \ref{thm:main}.

\section{\, Transonic Nozzles with General Cross-Sectional Areas}\label{sec:general}
As mentioned in the introduction,
the approach laid out above applies to much more general situations,
for example an infinitely narrowing nozzle  with closed ends or an expanding nozzle with unbounded ends.
In this section, we remove the additional assumptions placed on $A(x)$
in \S \ref{sec:artificial}--\S \ref{sec:convergence}
to conclude the proof of Theorem \ref{thm:main} for transonic nozzles with general cross-sectional area function
$A(x)$ satisfying \eqref{1.3a} or \eqref{1.3b}.

More precisely, we now relax the lower bound on $A$ to $A(x)>0$
and allow $A'$ to be only in $L^1(-\infty, 0)$ or $L^1(0,\infty)$,
especially allowing that the cross-sectional area function may converge to $0$ as $x\rightarrow\pm\infty$
and to $\infty$ as $x\to -\infty$ or $\infty$.
For ease of reference, we restate the main theorem, Theorem \ref{thm:main}, as follows:

\begin{theorem}\label{thm:5.1}
Assume that $(\rho_0,m_0)\in (L^1_{loc}(\mathbb{R}))^2$ with $\rho_0\geq0$
is of relative finite-energy and that the cross-sectional area function $A\in C^2$ satisfies \eqref{1.3a} or \eqref{1.3b}.
Then there exists a sequence of approximate solutions $(\rho^\varepsilon,m^\varepsilon)(t,x)$
solving \eqref{approx}--\eqref{BCs} for some initial data $(\rho_0^\varepsilon,m_0^\varepsilon)$ such that
$(\rho_0^\varepsilon,m_0^\varepsilon)\rightarrow(\rho_0,m_0)$ a.e. $x\in\mathbb{R}$ as $\varepsilon\rightarrow 0$,
with $(\rho_0^\varepsilon,m_0^\varepsilon)$ taken to be the zero extension of $(\rho_0^\varepsilon,m_0^\varepsilon)$
outside $(a,b)$,
and $(\rho^\varepsilon,m^\varepsilon)\rightarrow(\rho,m)$ for a.e. $(t,x)\in\mathbb{R}^2_+:=\mathbb{R}_+\times\mathbb{R}$
and in $L^p_{loc}(\mathbb{R}^2_+)\times L^q_{loc}(\mathbb{R}^2_+)$ for $p\in[1,\gamma+1)$ and $q\in[1,\frac{3(\gamma+1)}{\gamma+3})$,
as $\varepsilon\rightarrow 0$, such that $(\rho,m)(t,x)$ is a global finite-energy entropy solution of the transonic nozzle problem
\eqref{transonic}--\eqref{prob:Cauchy} for end-states $(\rho_\pm,m_\pm)$.
\end{theorem}

First we recall that, without loss of generality, we may assume $u_+=0$ by the Galilean invariance of the original multidimensional Euler equations.

As we continue to impose the Dirichlet boundary conditions for the approximate problems \eqref{approx},
the existence of the approximate solutions is obtained as before (observe that once we restrict attention to the interval $(a,b)$, we recover
$0<A_0(\varepsilon)\leq A(x)\leq A_1(\eps)<\infty$).
However, we then need to demonstrate more carefully the uniform energy estimates of Proposition \ref{energy} and of \S \ref{sec:uniform}
that enable us to employ the compensated compactness framework in \cite{ChenPerep1} to pass to the vanishing viscosity limit and obtain an entropy solution of the transonic
nozzle problem.
In order to make these estimates uniform with respect to $\varepsilon$, we choose $\delta(\varepsilon)$, $a(\varepsilon)$, and $b(\varepsilon)$
such that the following inequalities hold for a constant $M$ independent of $\varepsilon$:

  \begin{itemize}
   \item[\textbullet] $\,\, \varepsilon|b-a|\leq M$,

   \smallskip
   \item[\textbullet] $\,\, \varepsilon\|\big(\frac{A'}{A}\big)'\|_{L^\infty(a,b)}\|A\|_{L^\infty(a,b)}|b-a|\leq M$,

   \smallskip
   \item[\textbullet] $\,\, \varepsilon\|A''\|_{L^\infty(a,b)}\leq M$,

   \smallskip
   \item[\textbullet] $\,\, \delta\varepsilon^{-1}\|A\|_{L^\infty(a,b)}|a|^{\beta(\gamma)}\big\|A^{\frac{\gamma-3}{\gamma-1}}\big\|_{L^\infty(a,b)}\leq M$,

   \smallskip
   \item[\textbullet] $\,\, \delta\varepsilon^{-1}\|A\|_{L^\infty(a,b)}|a|\leq M$,

   \smallskip
   \item[\textbullet] $\,\, \delta\|A\|_{L^\infty(a,b)}|a|^2\big\|A^{-\frac{4}{2\gamma-4}}\big\|_{L^\infty(a,b)}\leq M$,
  \end{itemize}
where $\beta(\gamma)$ is determined in Lemma 4.3.

Such choices of the parameters can be achieved as long as the cross-sectional area function $A\in C^2$ satisfies \eqref{1.3a} or \eqref{1.3b}.

\subsection{\, Uniform Estimates}
 As mentioned, to establish Theorem \ref{thm:5.1} for more general nozzles, the first step is to obtain the uniform energy estimate,
 which is the subject of the following proposition.

 \begin{proposition}\label{lemma:energy-main-general}
 Let
 $$
 E_0:=\sup_{\varepsilon>0}\int_a^b\overline{\eta_\delta^*}(\rho_0^\varepsilon(x),m_0^\varepsilon(x))A(x)\,dx<\infty.
 $$
 Then there exists $M>0$, independent of $\varepsilon>0$, such that, for all $\varepsilon>0$,
 \beqa\label{ineq:energy-general}
 &\sup_{t\in[0,T]}\int_a^b\big(\tfrac{1}{2}\rho |u-\bar{u}|^2 + \overline{h_\delta}(\rho,\bar{\rho})\big)A(x)\,dx\\
 &\,\,\,\,
 +\varepsilon\int_{Q_T}\Big(h_\delta''(\rho)\rho_x^2
   +\rho u_x^2+\big|\big(\frac{A'(x)}{A(x)}\big)'\rho u(u-\bar{u})\big|\Big)A(x)\,dx\,dt\\
 &\,\,\leq M(E_0+1).
 \eeqa
 \end{proposition}

 \begin{proof}
 \, We prove estimate \eqref{ineq:energy-general} for the nozzles satisfying condition \eqref{1.3a}, since
the other case \eqref{1.3b} can be handled by corresponding similar arguments.

 Exactly as in the proof of Proposition \ref{energy},
 \begin{align*}
  &\frac{dE}{dt}-\frac{m}{\rho}p_\delta(\rho)\vert_a^b\\
  &\quad+\int_a^b \Big(m\big((\eta_\delta^*)_\rho(\bar{\rho},\bar{m})\big)_xA(x)
     +\big((\eta_\delta^*)_m(\bar{\rho},\bar{m})\big)_x\big(\frac{m^2}{\rho}+p_\delta(\rho)\big)A(x)\\
  &\qquad\qquad\, +(\eta_\delta^*)_m(\bar{\rho},\bar{m})A'(x)p_\delta(\rho) \\
  &\qquad\qquad\, +\varepsilon\big(\frac{A'(x)}{A(x)}\big)' m\big((\eta_\delta^*)_m-(\eta_\delta^*)_m(\bar{\rho},\bar{m})\big)\Big)\,dx\\
  &=-\varepsilon\int_a^b (\rho_x,m_x)\nabla^2\overline{\eta_\delta^*}(\rho_x,m_x)^\top A(x)\,dx\\
    &\quad+\varepsilon\int_a^b \big(\rho_x((\eta^*_\delta)_\rho(\bar{\rho},\bar{m}))_x+m_x((\eta^*_\delta)_m(\bar{\rho},\bar{m}))_x\big)A(x)\,dx.
 \end{align*}
 Observe now that
 \begin{align*}
  \varepsilon\int_a^b&\big(\rho_x((\eta^*_\delta)_\rho(\bar{\rho},\bar{m}))_x+m_x((\eta^*_\delta)_m(\bar{\rho},\bar{m}))_x\big) A(x)\, dx\\
  \leq& -\varepsilon\int_{-L_0}^{L_0}\big(\rho((\eta^*_\delta)_\rho(\bar{\rho},\bar{m}))_x+m((\eta^*_\delta)_m(\bar{\rho},\bar{m}))_x\big)A'(x)\,dx\\
  &-\varepsilon\int_{-L_0}^{L_0}\big(\rho((\eta^*_\delta)_\rho(\bar{\rho},\bar{m}))_{xx}+m((\eta^*_\delta)_m(\bar{\rho},\bar{m}))_{xx}\big)A(x)\,dx\\
  &+\varepsilon \big(\rho((\eta^*_\delta)_\rho(\bar{\rho},\bar{m}))_x+m((\eta^*_\delta)_m(\bar{\rho},\bar{m}))_x\big)A|^b_a\\
  \leq&\,\varepsilon M(\bar{\rho},\bar{m})\big(E+\|A'\|_{L^1(-L_0,L_0)}+\|A\|_{L^\infty(-L_0,L_0)}L_0\big)\\
  \leq&\, M(E+1),
 \end{align*}
 since $(\bar{\rho},\bar{m})$ are constant outside $(-L_0,L_0)$ so that both
 $\big((\eta^*_\delta)_\rho(\bar{\rho},\bar{m})\big)_x$ and $\big((\eta^*_\delta)_m(\bar{\rho},\bar{m})\big)_x$
 vanish outside this interval. We have also used that $|a(\varepsilon)|>L_0$ and $|b(\varepsilon)|>L_0$
 to see that the boundary
 term in the integration by parts vanishes. Similarly, we have
 \begin{align*}
 &\left|\int_a^bm\big(\big(\eta_\delta^*\big)_\rho(\bar{\rho},\bar{m})\big)_xA(x)\,dx\right|\leq M\|A\|_{L^\infty(-L_0,L_0)}L_0(E+1),\\
 &\left|\int_a^b\big(\big(\eta_\delta^*\big)_m(\bar{\rho},\bar{m})\big)_x\big(\frac{m^2}{\rho}+p_\delta(\rho)\big)A(x)\,dx\right|\leq M\|A\|_{L^\infty(-L_0,L_0)}L_0(E+1).
 \end{align*}
 Using the uniform bound $|\frac{A'}{A}|\leq M$, and $(\eta^*_\delta)_m(\bar{\rho},\bar{m})=0$ for $x>L_0$, we obtain
 \begin{align*}
  \left|\int_a^bp_\delta(\rho)(\eta^*_\delta)_m(\bar{\rho},\bar{m})A'(x)\, dx\right|
  \leq&\,\int_a^{L_0}\big(MA(x)\overline{h_\delta}(\rho,\bar{\rho})+|A'(x)|\big)\,dx\\
  \leq&\, ME+\|A'\|_{L^1(-\infty,L_0)}.
 \end{align*}
 Finally, we make the estimate:
 \begin{align*}
  \int_a^b&\varepsilon\big|\big(\frac{A'(x)}{A(x)}\big)'m\big((\eta_\delta^*)_m-(\eta_\delta^*)_m(\bar{\rho},\bar{m})\big)\big|\,dx\\
  \leq&\,\varepsilon M \|\big(\frac{A'}{A}\big)'\|_{L^\infty(a,b)}\int_a^b\big(\rho(u-\bar{u})^2+\rho\bar{u}^2\big)A(x)\,dx\\
  \leq&\,\varepsilon M\|\big(\frac{A'}{A}\big)'\|_{L^\infty(a,b)}\big(E+\|A\|_{L^\infty(a,b)}|b-a|\big)\\
  \leq&\, M(E+1).
 \end{align*}
With these estimates, we use the Gronwall inequality to obtain the result as before.
 \qed\end{proof}

We must now consider the key estimates of
Lemma \ref{unif2}.
Before that,
we first consider Lemma \ref{rhocubed}.
In the first line of the proof, we make the initial estimate:
\begin{align*}
  \varepsilon&\int_0^T\int_a^x\rho^3A(y)\,dy\,dt\leq \varepsilon |a| M \|A\|_{L^\infty(a,b)}\int_0^T\sup_{(a,x)}\rho^{3-\gamma}\,dt.
 \end{align*}
Continuing with the proof as before, we find that, by assumption,
$$
\varepsilon\int_0^T\int_a^x\rho^3\,dy\,dt
\leq M\varepsilon\|A\|_{L^\infty(a,b)}\big(|a|^{\beta(\gamma)}\|A^{\frac{\gamma-3}{\gamma-1}}\|_{L^\infty(a,b)}+1\big)
\leq M\frac{\varepsilon}{\delta}.
$$
The other case, $\gamma\in[2,3]$, requires a similar modification, leading to
\begin{align*}
\varepsilon\int_0^T\int_a^x&\rho^3A\,dy\,dt
\leq M\|A\|_{L^\infty(a,b)}|a|\big(\varepsilon+1+\varepsilon|a|\|A^{-\frac{4}{2\gamma-4}}\|_{L^\infty(a,b)}\big)
\leq M\frac{\varepsilon}{\delta}.
\end{align*}
Finally, we examine the proof of Lemma \ref{unif2}.
Turning attention to estimates \eqref{ineq:integral1}--\eqref{ineq:integral3},
we see that they remain uniform, under the assumption that $\varepsilon\|A''\|_{L^\infty(a,b)}\leq M$ and
the uniform bound of Lemma \ref{rhocubed}.

We need to make the bounds on all of terms $I_j, j=1,\dots, 7$, uniform in $\varepsilon$ in order to conclude the expected result.
Note first that $I_1$ and $I_3$ need no adjustment.
We have seen that
$$
\int_K\omega|I_2|\,dx\leq M\varepsilon|a|\big\|A\big(\frac{A'}{A}\big)'\big\|_{L^\infty(a,b)}\leq M,
$$
by assumption.
Examining the estimate for $I_4$ shows that
\begin{align*}
|I_4|\leq&\frac{\delta}{\sqrt{\varepsilon}}M\Big(\int_0^T\int_a^x\rho^3(t,y)A(y)\,dy\,dt\Big)^{\frac{1}{2}}\leq M  \qquad \mbox{uniformly in $\varepsilon$}.
\end{align*}

The bounds for $I_5$ and $I_6$ are seen to be uniform,
due to the uniform bound of Lemma \ref{rhocubed} and $|\frac{A'}{A}|\leq M$.
Finally, we have
\begin{align*}
\int_K|I_7|\,dx=\big\|\frac{A'}{A}\big\|_{L^\infty}\int_0^T\int_K\omega(x)|\tilde{\eta}(t,x)|A(x)\,dx\,dt\leq M.
\end{align*}

\subsection{\, Compactness}
With the uniform estimates discussed above, the only thing remaining to prove
is Proposition \ref{compactness}.
Its proof is seen to follow as before by using the uniform bound $|\frac{A'}{A}|\leq M$
and by observing that, in each of terms $J_1, j=1, \dots, 6$, all the integrals are taken over the support of the test function $\omega$,
so that the factors of $A$, $A^{-1}$, {\it etc.},  may be introduced freely.
This concludes the proof of Theorem 6.1, hence also Theorem 2.1.

\section{\, Spherically Symmetric Solutions to the Multidimensional Euler Equations}

In this section, we explain how the approach and techniques developed in \S 3-- \S 6 above
can be generalized to construct global spherically symmetric solutions to
the multidimensional Euler equations for compressible fluid flows, as considered in Chen-Perepelitsa in \cite{ChenPerep2}.
This situation corresponds to the case that $A(x)=\omega_n x^{n-1}: \mathbb{R}_+\rightarrow\mathbb{R}_+$, where $\omega_n$ is
the surface area of the $n$-dimensional unit sphere.

In \cite{ChenPerep2},
the possibility of a blow-up of the density at the origin is allowed
by imposing a Neumann boundary condition
at the left boundary for the approximate solutions.
By the Galilean invariance, the velocity can be taken to vanish at $x=\infty$,
so the end-states $(\rho_+,u_+)=(\bar\rho,0)$ at $b(\eps)$ are imposed,
where $\bar\rho\rightarrow 0$ as $\varepsilon \rightarrow 0$.
Then the solutions of this approximate problem were constructed under the assumption that $\gamma\in(1,3]$,
while the convergence of the approximate solutions to the limit is shown
for all $\gamma>1$ in  \cite{ChenPerep2}.

The following theorem gives the extension of this result for all the physical interval $\ga>1$,
especially including the unsolved case $\gamma> 3$,
which leads to Theorem \ref{thm:sph-symm}.

\begin{theorem}\label{thm:sph-symm-redux}
Let $(\rho_0,m_0)\in (L^1_{loc}(\R_+))^2$ be finite-energy initial data such that $\rho_0\geq0$.
Then there exists a global finite-energy entropy solution of the spherically symmetric Euler equations{\rm :}
$$
\begin{cases}
\rho_t+ m_x+\tfrac{n-1}{x}m=0,& (t,x)\in\mathbb{R}_+\times\mathbb{R}_+,\\
m_t+ \big(\tfrac{m^2}{\rho}+p(\rho)\big)_x+\tfrac{n-1}{x}\tfrac{m^2}{\rho}=0,& (t,x)\in\mathbb{R}_+\times\mathbb{R}_+,\\
(\rho,m)|_{t=0}=(\rho_0,m_0),& x\in \R_+,
\end{cases}
$$
where $p(\rho)=\kappa\rho^\gamma$ with $\gamma>1$.
This global entropy solution is the vanishing viscosity limit of the approximate solutions of
the following system{\rm :}
\beqa\label{eq:sph-symm-approx}
\begin{cases}
\rho^\eps_t+ m^\eps_t+\tfrac{n-1}{x}m^\eps=\eps x^{-(n-1)}\big(x^{n-1}\rho^\eps_x\big)_x,& (t,x)\in Q^\eps,\\
m^\eps_t+\big(\tfrac{(m^\eps)^2}{\rho^\eps}+p_\de(\rho^\eps)\big)_x+\tfrac{n-1}{x}\tfrac{(m^\eps)^2}{\rho^\eps}=\eps\big(m^\eps_x+\tfrac{n-1}{x}m^\eps\big)_x,& (t,x)\in Q^\eps,\\
(\rho^\eps,m^\eps)|_{t=0}=(\rho_0^\eps,m_0^\eps),& x\in(a,b),\\
\end{cases}
\eeqa
where $Q^\eps=\R_+\times(a(\eps),b(\eps))$ and with appropriately chosen boundary conditions{\rm ;} that is, as $\varepsilon\to 0$,
$$
(\rho^\varepsilon,m^\varepsilon)\rightarrow(\rho,m) \qquad\,\mbox{for a.e.} \, (t,x)\in\mathbb{R}^2_+
$$
and in $L^p_{loc}(\mathbb{R}^2_+)\times L^q_{loc}(\mathbb{R}^2_+)$ for $p\in[1,\gamma+1)$ and $q\in[1,\frac{3(\gamma+1)}{\gamma+3})$.
Here $p_\de(\rho)=\kappa\rho^\gamma+\de\rho^2$, $a(\eps)\to0$ and $b(\eps)\to\infty$ as $\eps\to0$, and $(\rho_0^\varepsilon,m_0^\varepsilon)\rightarrow(\rho_0,m_0)$
a.e. $x\in\mathbb{R}$ as $\varepsilon\rightarrow 0$, where $(\rho_0^\varepsilon,m_0^\varepsilon)$ are taken to be the zero extension of $(\rho_0^\varepsilon,m_0^\varepsilon)$
 outside $(a(\eps),b(\eps))$.
\end{theorem}

To overcome the obstacle encountered in \cite{ChenPerep2} for the unsolved case $\gamma\ge 3$,
we develop two approaches here, thereby providing two methods for re-solving the problem
for the spherically symmetric solutions
for the whole range $\gamma\in (1, \infty)$.
In \S \ref{subsec:Dirichlet}, instead of taking the Neumann data for $\rho$ at $a(\eps)$ as in  \cite{ChenPerep2},
we choose the Dirichlet data at both ends here.
One of the motivations is the fact that
the boundary data are allowed not to be preserved in the limit (as boundary layers)
by choosing $(\rho,u)|_{x=a}=(\bar\rho,0)$ for the same $\bar\rho$
as at the other end-point in the construction of the approximate solutions.
This allows us to use the above construction from \S \ref{sec:artificial}
(especially the maximum principle for the approximate equations) to demonstrate
the existence of the approximate solutions.

In \S \ref{subsec:Neumann}, we consider the same problem but with the Neumann boundary data
initiated in \cite{ChenPerep2}. We are able to demonstrate how the higher order \it a priori \rm energy estimates
 may be obtained for the approximate solutions even in this situation to
 conclude the existence of the approximate solutions for all $\ga>1$, especially including the unsolved case $\gamma>3$,
 concluding the scheme set out in \cite{ChenPerep2}.

\subsection{\, Dirichlet Boundary Conditions for the Density}\label{subsec:Dirichlet}
We first demonstrate how the imposition of the Dirichlet boundary conditions
for the approximate system \eqref{eq:sph-symm-approx} may be used to obtain
the existence of globally defined, spherically symmetric entropy solutions of  the compressible Euler equations.
We therefore assign the Dirichlet boundary conditions to system \eqref{eq:sph-symm-approx}:
\beq\label{eq:sph-symm-dirichlet-bc}
(\rho,m)|_{x=a}=(\bar\rho,0),\,\qquad (\rho,m)|_{x=b}=(\bar\rho,0),
\eeq
where $\bar\rho=\bar{\rho}(\eps)\to0$ as $\eps\to0$.

Once again, the main point to check is that the key energy estimate holds
for the approximate solutions.
Indeed, the imposition of the Dirichlet data (rather than the Neumann data) enables us
to apply the framework in \S \ref{sec:artificial} directly to construct the approximate
solutions.
Moreover, by choosing the same boundary data at the two end-points,
we may take the monotone reference functions $(\bar\rho(x),\bar u(x))$ to be the constant state $(\bar\rho,0)$.
Then the relative mechanical entropy becomes
$$
\overline{\eta_\de^*}(\rho,m)=\eta_\de^*(\rho,m)-\eta_\de^*(\bar\rho,0)-(\eta_\de^*)_\rho(\bar\rho,0)(\rho-\bar\rho).
$$

\begin{proposition}
Let $$E_0:=\sup_{\varepsilon>0}\int_a^b\overline{\eta_\de^*}(\rho_0^\varepsilon(x),m_0^\varepsilon(x))A(x)\,dx<\infty.$$
Then there exists $M>0$, independent of $\varepsilon$, such that, for all $\varepsilon>0$,
 \begin{equation}\label{ineq:energy-spherical}
  \begin{split}
   &\sup_{t\in[0,T]}\int_a^b\big(\tfrac{1}{2}\rho u^2 + \overline{h_\delta}(\rho,\bar\rho)\big)x^{n-1}\,dx\\
   &+\varepsilon\int_{Q_T}\big(h_\delta''(\rho)\rho_x^2+\rho u_x^2+(n-1)\frac{\rho u^2}{x^2}\big)x^{n-1}\,dx\,dt\leq ME_0.
  \end{split}
 \end{equation}
\end{proposition}

\begin{proof}
\, With the above choice of $A(x)=\omega_n x^{n-1}$, equation \eqref{eq:star} is simplified as
\beqas
(\overline{\eta_\de^*}x^{n-1})_t+\big((q_\de^*-(\eta_\de^*)_\rho(\bar\rho,0)\big)_x+\eps(n-1)m(\eta_\de^*)_mx^{n-3}\nonumber  \\[1mm]
=\eps(\rho_xx^{n-1})_x\big((\eta_\de^*)_\rho-(\eta_\de^*)_\rho(\bar\rho,0)\big)+\eps(m_xx^{n-1})_x(\eta_\de^*)_m. \nonumber
\eeqas
Integrating this over the parabolic cylinder $Q_T$ and then integrating by parts with the new Dirichlet boundary conditions implies
\beqas
\int_a^b\overline{\eta_\de^*}x^{n-1}\,dx+\eps\int_{Q_T}\big((\rho_x,m_x)\nabla^2\overline{\eta_\de^*}(\rho_x,m_x)^\top+\frac{n-1}{x^2}\frac{m^2}{\rho}\big)x^{n-1}\,dx\,dt=E_0, \nonumber
\eeqas
as desired. Then we conclude as in Proposition \ref{energy} by the convexity of $\overline{\eta_\de^*}$.
\qed\end{proof}

With this estimate, we are now in the situation in \S \ref{sec:artificial} of this paper.
We can therefore proceed with the maximum principle estimate of Lemma \ref{lemma:max-principle}
to obtain {\it a priori} upper bounds on the density and velocity of the fluid,
and continue to deduce the existence of solutions $(\rho^\eps,m^\eps)$ of the initial-boundary value problem
\eqref{eq:sph-symm-approx}--\eqref{eq:sph-symm-dirichlet-bc},
exactly as in Theorem \ref{artsol}:

\begin{lemma}\label{lemma:sph-symm-approx}
For $\varepsilon>0$, let $(\rho_0^\varepsilon,m_0^\varepsilon)\in\big(C^{2+\beta}([a,b])\big)^2$ be a sequence
of functions such that
  \begin{enumerate}
   \item[\rm (i)] $\inf_{a\leq x\leq b}\rho_0^\varepsilon(x)>0${\rm ;}

   \smallskip
   \item[\rm (ii)] $(\rho_0^\varepsilon,m_0^\varepsilon)$ satisfies $\eqref{eq:sph-symm-dirichlet-bc}$ and the compatibility conditions{\rm :}
   \beqa
   &\big(x^{n-1}m_0^\eps\big)_x\big|_{x=a}=0, \nonumber  \\
   & m_{0,x}^\eps|_{x=b}=\varepsilon x^{-(n-1)}\big(x^{n-1}\rho_{0,x}^\eps\big)_x\big|_{x=b}, \nonumber  \\
   &\Big(\frac{(m_0^\eps)^2}{\rho_0^\eps}+p_\de(\rho_0^\eps)\Big)_x\big|_{x=b}=\varepsilon x^{-(n-1)}\big(x^{n-1}m_0^\eps\big)_x\big|_{x=b}; \nonumber
   \eeqa
   \item[\rm (iii)] $\int_a^b\big(\frac{(m_0^\varepsilon)^2}{2\rho_0^\varepsilon}+\frac{\kappa(\rho_0^\varepsilon)^\gamma}{\gamma-1}\big)x^{n-1}\,dx<\infty$.
   \end{enumerate}
 Then there exists a unique global solution $(\rho^\varepsilon,m^\varepsilon)$ of \eqref{eq:sph-symm-approx}--\eqref{eq:sph-symm-dirichlet-bc} for
 $\gamma\in(1,\infty)$ such that $(\rho^\varepsilon,m^\varepsilon)\in\big(C^{2+\beta,1+\frac{\beta}{2}}(Q_T)\big)^2$ with $\inf_{Q_T}\rho^\varepsilon(t,x)>0$
 for all $T>0$.
\end{lemma}

To conclude the proof of Theorem \ref{thm:sph-symm-redux}, we note that, as remarked in \cite{ChenPerep2},
the arguments to derive the uniform estimates and convergence of approximate solutions in  \cite{ChenPerep2}
are independent of the choice of $\ga\in(1,3]$, but rather work for all $\ga\in(1,\infty)$.
In addition, we now observe that, in obtaining these uniform estimates, the integrals were considered only over compact
regions $K\subset\R_+$, or over intervals $(x,b)$ with $x\in K$ for a compact region $K\subset\R_+$.
Therefore, our modification of the boundary condition at $a(\eps)$ in the above
does not affect these estimates at all.
In particular, we have the following lemma:

\begin{lemma}
Suppose that $(\rho_0^\eps, m_0^\eps)$ satisfy the assumptions of Lemma {\rm \ref{lemma:sph-symm-approx}} and that
$$
\sup_\eps \int_a^b\Big(\frac{(m_0^\varepsilon)^2}{2\rho_0^\varepsilon}+\frac{\kappa(\rho_0^\varepsilon)^\gamma}{\gamma-1}\Big)x^{n-1}\,dx<\infty.
$$
Let $\de(\eps),a(\eps),b(\eps)$, and $\bar\rho(\eps)$ satisfy the following relation{\rm :}
$$
\bar\rho^\ga b^n+\frac{\de}{\eps}b^n\leq M,
$$
where $M$ is a constant independent of $\eps$.
Then, for any compact region $K\subset\R_+$, there exists $M>0$ independent of $\eps$ such that
\beq
\int_0^T\int_K\big((\rho^\eps)^{\ga+1}+\de(\rho^\eps)^3+\rho^\eps|u^\eps|^3+(\rho^\eps)^{\ga+\th}\big)\,dx\,dt\leq M. \nonumber
\eeq
In addition, if $\psi(s)$ is any smooth, compactly supported test function on $\R$,
then the associated entropy pair $(\eta^\psi,q^\psi)$ satisfies
\beq
\eta^\psi(\rho^\eps,m^\eps)_t+q^\psi(\rho^\eps,m^\eps)_x \qquad \text{ are compact in } H^{-1}_{loc}. \nonumber
\eeq
\end{lemma}

The convergence of the approximate solutions to an entropy solution of the spherically symmetric Euler equations
now follows from the standard argument in \cite{ChenPerep2}. This completes the proof of Theorem 7.1, hence Theorem \ref{thm:sph-symm}.

\subsection{\, Neumann boundary conditions for the density}\label{subsec:Neumann}
Finally, we consider the approximate solutions
of system \eqref{eq:sph-symm-approx} with the boundary data given by
\beq\label{eq:sph-symm-neumann-bc}
(\rho_x,m)|_{x=a}=(0,0),\qquad\, (\rho,m)|_{x=b}=(\bar\rho,0).
\eeq
This choice of Neumann data for the density at $x=a$ has the advantage that it allows the density
of the approximate solutions to become very large near the origin,
a motivation for its choice in \cite{ChenPerep2}.

From \cite[Proposition 2.1]{ChenPerep2}, we see that
the mechanical energy is uniformly bounded:

\begin{lemma}
Suppose that $(\rho,u)$ is a $C^{2,1}(Q_T)$ solution of problem \eqref{eq:sph-symm-approx} and \eqref{eq:sph-symm-neumann-bc}.
Then, for each $\eps>0$,
\beqa
&\sup_{[0,T]}\int_a^b\big(\half\rho u^2+\overline{h_\de}(\rho,\bar\rho)\big)x^{n-1}\, dx\nonumber\\
&+\eps\int_{Q_T}\big(h_\de''(\rho)|\rho_x|^2+\rho|u_x|^2+(n-1)\frac{\rho u^2}{x^2}\big)x^{n-1}\,dx\,dt\leq E_0.\nonumber
\eeqa
\end{lemma}

Observe now that $\overline{h_\de}(\rho,\bar\rho)$ grows, for $\rho$ large, as $\rho^\gamma$.
In order to close the estimates of \cite[Lemma 2.3]{ChenPerep2} for all $\ga>1$,
we require an additional higher order integrability of the density $\rho$.
To this end, we take the test function $\psi(s)=s^4$ and consider the entropy $\eta^\psi$ generated from it.
We observe that, by a simple calculation, $\eta^\psi$ satisfies
\beq
\rho u^4+\rho^{2\ga-1}\leq M\eta^\psi(\rho,m) \qquad \text{ for all } (\rho,m)\in\mathbb{R}^2_+,
\eeq
where the constant $M$ depends only on $\ga>1$.

\begin{lemma}\label{lemma:sph-symm-neumann-extra-energy}
Let
$$
\mathcal{E}_0:=\int_a^b\eta^\psi(\rho_0^\eps,m_0^\eps)x^{n-1}\,dx<\infty.
$$
Then solutions $(\rho^\eps,m^\eps)$ of \eqref{eq:sph-symm-approx} and \eqref{eq:sph-symm-neumann-bc} satisfy that,
for any $T>0$,
\beq
 \sup_{[0,T]}\int_a^b\eta^\psi(\rho^\eps,m^\eps)x^{n-1}\,dx\leq \mathcal{E}_0.
\eeq
In particular, there exists $C=C(\eps)$ such that
\beq\label{est:rho-two-gamma}
\sup_{[0,T]}\int_a^b\rho^{2\ga-1} x^{n-1}\,dx\leq C.
\eeq
\end{lemma}

The proof of this lemma follows from the same calculation as that of \cite[Proposition 2.1]{ChenPerep2},
 but with the new entropy $\eta^\psi$ in place of $\eta^*$. Observe that $\eta^\psi$ is also convex.

We are now in a position to see that the theorem regarding the existence of the approximate solutions
for the spherically symmetric Euler equations in \cite[Theorem 2.1]{ChenPerep2} (under the assumption of the Neumann boundary data for density $\rho$)
can be extended to the full range of $\gamma$-law gases, especially the case: $\gamma>3$.
We observe that it suffices to show only Lemma 2.3 in \cite{ChenPerep2}, since the rest of the argument follows as
in the case $\ga\in(1,3]$ considered already in \cite{ChenPerep2}.
We therefore provide a sketch of the proof of this result, highlighting the places where this new estimate is used.

\medskip
We aim to show for the whole range $\gamma\in (1,\infty)$ that there exists $C$, depending on $\eps>0$, such that the approximate solutions $(\rho^\eps,m^\eps)$ satisfy
\beq
\sup_{[0,T]}\int_a^b\big(|\rho^\eps_x|^2+|m^\eps_x|^2\big)\,dx+\int_{Q_T}\big(|\rho^\eps_{xx}|^2+|m^\eps_{xx}|^2\big)\,dx\,dt\leq C. \nonumber
\eeq

Proceeding as in the proof of \cite[Lemma 2.3]{ChenPerep2} and Lemma \ref{energy4}, we obtain
\beqa
\frac{1}{2}&\int_a^b \big(|\rho_x(t,x)|^2+|m_x(t,x)|^2\big)\,dx+\varepsilon \int_{Q_t}\big(|\rho_{xx}|^2+|m_{xx}|^2\big)\,dx\,d\tau \nonumber\\
=&\,\frac{1}{2}\int_a^b\big(|\rho_{0,x}|^2+|m_{0,x}|^2\big)\,dx+\int_{Q_t}\big(m_x\rho_{xx}+\frac{n-1}{x}m\rho_{xx}\big)\,dx\,d\tau\\
&\,+\int_{Q_t}(\rho u^2+p_\de)_xm_{xx}\,dx\,d\tau\nonumber \\
&\,+(n-1)\int_{Q_t}\big(\frac{1}{x}\rho u^2 m_{xx}-\frac{\eps}{x}\rho_x\rho_{xx}\big)\,dx\,d\tau
  -(n-1)\eps\int_{Q_t}\Big(\frac{m}{x}\Big)_xm_{xx}\,dx\,d\tau. \nonumber
\eeqa
We here give only the estimate of the critical term $(\rho u^2)_xm_{xx}=u^2\rho_xm_{xx}+2\rho uu_xm_{xx}$.
Following \cite[Lemma 2.3]{ChenPerep2}, we obtain that, for $\De>0$ to be chosen later,
\beqas
\int_{Q_t}|u^2\rho_xm_{xx}|\,dx\,d\tau
\leq&\, \De\int_{Q_t}|m_{xx}|^2\,dx\,d\tau\\
 &\, +C_\De\int_0^t\Big(\|u(\tau,\cdot)\|_{L^\infty}^4\int_a^bh_\de''(\rho)|\rho_x(\tau,x)|^2\,dx\Big)\,d\tau.
\eeqas
From the maximum principle lemma, \cite[Lemma 2.2]{ChenPerep2},
$$
\|u(\tau,\cdot)\|_{L^\infty}^4\leq C\Big(1+\|\rho\|_{L^\infty(Q_\tau)}^{2\max\{1,\ga-1\}}\Big),
$$
which yields
\beqas
&\int_{Q_t}|u^2\rho_xm_{xx}|\,dx\,d\tau\\
&\leq\De\int_{Q_t}|m_{xx}|^2\,dx\,d\tau\\
  &\quad+C_\De\int_0^t\Big(\big(1+\sup_{s\in[0,\tau]}\|\rho(s,\cdot)\|_{L^\infty}^{2\max\{1,\ga-1\}}\big)\int_a^bh_\de''(\rho)|\rho_x(\tau,x)|^2\,dx\Big)\,d\tau.
\eeqas
Now we can  use the improved estimate \eqref{est:rho-two-gamma} and the argument of \cite[Lemma 2.1]{ChenPerep2},
and also compare with Lemma \ref{energy2} to show that
$$
\|\rho^{2\ga+1}\|_{L^\infty(a,b)}\leq C\Big(1+\int_a^b|\rho_y|^2dy\Big).
$$
We therefore obtain
\beqa
&\int_{Q_t}|u^2\rho_xm_{xx}|\,dx\,d\tau\nonumber\\
&\leq\De\int_{Q_t}|m_{xx}|^2\,dx\,d\tau\\
 &\quad+C_\De\int_0^t\Big(\big(1+\sup_{s\in[0,t]}\int_a^b|\rho_x(s,x)|^2\,dx\big)\int_a^bh_\de''(\rho)|\rho_x(\tau,x)|^2\,dx\Big)\,d\tau \nonumber
\eeqa
for the whole range $\gamma\in (1, \infty)$, especially including $\gamma>3$.

Arguing similarly for the other terms, we conclude the result.

This result now allows us to conclude the argument of Theorem 2.1 in \cite{ChenPerep2}
as it is done there.
As remarked previously, this is sufficient to
apply the results of \S 3 in \cite{ChenPerep2}
to conclude Theorem \ref{thm:sph-symm-redux}, and hence Theorem \ref{thm:sph-symm}.

%

\medskip

\begin{acknowledgements}
\, The research of
Gui-Qiang G. Chen was supported in part by
the UK
Engineering and Physical Sciences Research Council Award
EP/E035027/1 and
EP/L015811/1, and the Royal Society--Wolfson Research Merit Award (UK).
The research of Matthew R. I. Schrecker  was supported in part by
the UK Engineering and Physical Sciences Research Council Award
EP/L015811/1.
The authors would also like to thank Feimin Huang and Mikhail Perepelitsa for helpful discussions.
\end{acknowledgements}



\end{document}